\documentclass[10pt]{amsart}
\usepackage{import}
\usepackage{mathabx}
\usepackage{bbm}
\usepackage{fontenc}
\usepackage[T1]{fontenc}
\usepackage{amsfonts}
\usepackage{mathrsfs}
\usepackage{stmaryrd}
\usepackage{upgreek}
\usepackage{textcomp}

\usepackage{extarrow}
\usepackage{enumerate}
\usepackage{graphicx}
\usepackage[colorlinks=true, linkcolor=blue]{hyperref}
\usepackage{latexsym}
\usepackage{mathtools}
\usepackage{mathtext}
\usepackage{amsmath}
\usepackage{amssymb}
\usepackage{amsthm}
\usepackage{tikz}
\usepackage[all,cmtip]{xy}
\usepackage[mathscr]{eucal}
\usepackage[toc,page]{appendix}
\usepackage{a4wide}

\usetikzlibrary{arrows}
\usetikzlibrary{matrix}
\usetikzlibrary{shapes}
\usetikzlibrary{snakes}
\usetikzlibrary{matrix}

\DeclareMathOperator{\Coh}{\textup{Coh}}

\DeclareMathOperator{\GL}{\textup{GL}}

\DeclareMathOperator{\Homsh}{\underline{\textup{Hom}}}

\DeclareMathOperator{\Spec}{\textup{Spec}}

\DeclareMathOperator{\et}{\textup{et}}

\usepackage[mathscr]{eucal}
\usepackage{enumitem}
\usepackage{xypic}
\usepackage{transparent,comment}

\usepackage{biblatex}
\theoremstyle{plain}
\newtheorem{thm}{Theorem}[section]
\newtheorem*{thm*}{Theorem}
\newtheorem{lem}[thm]{Lemma}
\newtheorem{cor}[thm]{Corollary}
\newtheorem{prop}[thm]{Proposition}

\theoremstyle{remark}
\newtheorem{rmk}[thm]{Remark}

\newtheorem*{rmk*}{Remark}
\newtheorem{ex}[thm]{Example}

\theoremstyle{definition}
\newtheorem{defn}[thm]{Definition} 
\newtheorem{const}[thm]{Construction}
\newtheorem*{const*}{Construction}
\newtheorem{conv}[thm]{Conventions}

\newtheorem{sett}[thm]{Setting}

\theoremstyle{plain}
\newtheorem{thmI}{Theorem}
\newtheorem{thmII}{Theorem}
\newtheorem{thmIII}{Theorem}
\newtheorem{thmIV}{Theorem}

\def\hatX{{\hat X}}
\def\sYrho{{\hat Y}}

\begin{document}
\title[Meromorphic descent and the pro-\'etale fundamental group]{A theorem on meromorphic descent and the specialization of the pro-\'etale fundamental group}
\author{Marcin Lara, Jiu-Kang Yu, Lei Zhang}

\address{Marcin LARA\\ Instytut Matematyczny PAN\\ Śniadeckich 8\\ Warsaw, Poland}
\email{marcin.lara@impan.pl}

\address{ 
Jiu-Kang YU\\
     The Chinese University of Hong Kong\\
    Department of Mathematics\\    
    Shatin, New Territories\\ Hong Kong }
\email{jkyu@ims.cuhk.edu.hk}

 \address{Lei ZHANG\\
    The Chinese University of Hong Kong\\
    Department of Mathematics\\    
    Shatin, New Territories\\ Hong Kong }
\email{lzhang@math.cuhk.edu.hk} 

\date{\today}

\global\long\def\A{\mathbb{A}}

\global\long\def\Ab{(\textup{Ab})}

\global\long\def\C{\mathbb{C}}

\global\long\def\Cat{(\textup{Cat})}

\global\long\def\Di#1{\textup{D}(#1)}

\global\long\def\E{\mathbb{E}}

\global\long\def\F{\mathbb{F}}

\global\long\def\GCov{G\textup{-Cov}}

\global\long\def\Gcat{(\textup{Galois cat})}

\global\long\def\Gfsets#1{#1\textup{-fsets}}

\global\long\def\Gm{\mathbb{G}_{m}}

\global\long\def\GrCov#1{\textup{D}(#1)\textup{-Cov}}

\global\long\def\Grp{(\textup{Grps})}

\global\long\def\Gsets#1{(#1\textup{-sets})}

\global\long\def\HCov{H\textup{-Cov}}

\global\long\def\MCov{\textup{D}(M)\textup{-Cov}}

\global\long\def\MHilb{M\textup{-Hilb}}

\global\long\def\N{\mathbb{N}}

\global\long\def\PGor{\textup{PGor}}

\global\long\def\PGrp{(\textup{Profinite Grp})}

\global\long\def\PP{\mathbb{P}}

\global\long\def\Pj{\mathbb{P}}

\global\long\def\Q{\mathbb{Q}}

\global\long\def\RCov#1{#1\textup{-Cov}}

\global\long\def\RR{\mathbb{R}}

\global\long\def\Sch{\textup{Sch}}

\global\long\def\WW{\textup{W}}

\global\long\def\Z{\mathbb{Z}}

\global\long\def\acts{\curvearrowright}

\global\long\def\alA{\mathscr{A}}

\global\long\def\alB{\mathscr{B}}

\global\long\def\arr{\longrightarrow}

\global\long\def\arrdi#1{\xlongrightarrow{#1}}

\global\long\def\catC{\mathscr{C}}

\global\long\def\catD{\mathscr{D}}

\global\long\def\catF{\mathscr{F}}

\global\long\def\catG{\mathscr{G}}

\global\long\def\comma{,\ }

\global\long\def\covU{\mathcal{U}}

\global\long\def\covV{\mathcal{V}}

\global\long\def\covW{\mathcal{W}}

\global\long\def\duale#1{{#1}^{\vee}}

\global\long\def\fasc#1{\widetilde{#1}}

\global\long\def\fsets{(\textup{f-sets})}

\global\long\def\iL{r\mathscr{L}}

\global\long\def\id{\textup{id}}

\global\long\def\la{\langle}

\global\long\def\odi#1{\mathcal{O}_{#1}}

\global\long\def\ra{\rangle}

\global\long\def\set{(\textup{Sets})}

\global\long\def\sets{(\textup{Sets})}

\global\long\def\shA{\mathcal{A}}

\global\long\def\shB{\mathcal{B}}

\global\long\def\shC{\mathcal{C}}

\global\long\def\shD{\mathcal{D}}

\global\long\def\shE{\mathcal{E}}

\global\long\def\shF{\mathcal{F}}

\global\long\def\shG{\mathcal{G}}

\global\long\def\shH{\mathcal{H}}

\global\long\def\shI{\mathcal{I}}

\global\long\def\shJ{\mathcal{J}}

\global\long\def\shK{\mathcal{K}}

\global\long\def\shL{\mathcal{L}}

\global\long\def\shM{\mathcal{M}}

\global\long\def\shN{\mathcal{N}}

\global\long\def\shO{\mathcal{O}}

\global\long\def\shP{\mathcal{P}}

\global\long\def\shQ{\mathcal{Q}}

\global\long\def\shR{\mathcal{R}}

\global\long\def\shS{\mathcal{S}}

\global\long\def\shT{\mathcal{T}}

\global\long\def\shU{\mathcal{U}}

\global\long\def\shV{\mathcal{V}}

\global\long\def\shW{\mathcal{W}}

\global\long\def\shX{\mathcal{X}}

\global\long\def\shY{\mathcal{Y}}

\global\long\def\shZ{\mathcal{Z}}

\global\long\def\st{\ | \ }

\global\long\def\stA{\mathcal{A}}

\global\long\def\stB{\mathcal{B}}

\global\long\def\stC{\mathcal{C}}

\global\long\def\stD{\mathcal{D}}

\global\long\def\stE{\mathcal{E}}

\global\long\def\stF{\mathcal{F}}

\global\long\def\stG{\mathcal{G}}

\global\long\def\stH{\mathcal{H}}

\global\long\def\stI{\mathcal{I}}

\global\long\def\stJ{\mathcal{J}}

\global\long\def\stK{\mathcal{K}}

\global\long\def\stL{\mathcal{L}}

\global\long\def\stM{\mathcal{M}}

\global\long\def\stN{\mathcal{N}}

\global\long\def\stO{\mathcal{O}}

\global\long\def\stP{\mathcal{P}}

\global\long\def\stQ{\mathcal{Q}}

\global\long\def\stR{\mathcal{R}}

\global\long\def\stS{\mathcal{S}}

\global\long\def\stT{\mathcal{T}}

\global\long\def\stU{\mathcal{U}}

\global\long\def\stV{\mathcal{V}}

\global\long\def\stW{\mathcal{W}}

\global\long\def\stX{\mathcal{X}}

\global\long\def\stY{\mathcal{Y}}

\global\long\def\stZ{\mathcal{Z}}

\global\long\def\then{\ \Longrightarrow\ }

\global\long\def\L{\textup{L}}

\global\long\def\l{\textup{l}}

\newcommand{\B}{{\mathbb B}}
\newcommand{\D}{{\mathbb D}}
\newcommand{\G}{{\mathbb G}}
\renewcommand{\H}{{\mathbb H}}
\newcommand{\I}{{\mathbb I}}
\newcommand{\J}{{\mathbb J}}
\newcommand{\M}{{\mathbb M}}
\renewcommand{\P}{{\mathbb P}}
\newcommand{\R}{{\mathbb R}}
\newcommand{\T}{{\mathbb T}}
\newcommand{\U}{{\mathbb U}}
\newcommand{\V}{{\mathbb V}}
\newcommand{\W}{{\mathbb W}}
\newcommand{\X}{{\mathbb X}}
\newcommand{\Y}{{\mathbb Y}}

\newcommand{\sA}{{\mathcal A}}
\newcommand{\sB}{{\mathcal B}}
\newcommand{\sC}{{\mathcal C}}
\newcommand{\sD}{{\mathcal D}}
\newcommand{\sE}{{\mathcal E}}
\newcommand{\sF}{{\mathcal F}}
\newcommand{\sG}{{\mathcal G}}
\newcommand{\sH}{{\mathcal H}}
\newcommand{\sI}{{\mathcal I}}
\newcommand{\sJ}{{\mathcal J}}
\newcommand{\sK}{{\mathcal K}}
\newcommand{\sL}{{\mathcal L}}
\newcommand{\sM}{{\mathcal M}}
\newcommand{\sN}{{\mathcal N}}
\newcommand{\sO}{{\mathcal O}}
\newcommand{\sP}{{\mathcal P}}
\newcommand{\sQ}{{\mathcal Q}}
\newcommand{\sR}{{\mathcal R}}
\newcommand{\sS}{{\mathcal S}}
\newcommand{\sT}{{\mathcal T}}
\newcommand{\sU}{{\mathcal U}}
\newcommand{\sV}{{\mathcal V}}
\newcommand{\sW}{{\mathcal W}}
\newcommand{\sX}{{\mathcal X}}
\newcommand{\sY}{{\mathcal Y}}
\newcommand{\sZ}{{\mathcal Z}}


\newcommand{\Aff}{{\rm Aff}}
\newcommand{\Aut}{{\rm Aut}}
\newcommand{\an}{{\rm an}}
\newcommand{\Bd}{{\rm Band}}
\newcommand{\Cats}{{\rm Cats}}
\newcommand{\ch}{\textup{Ch}}
\newcommand{\Char}{{\rm char}}
\newcommand{\codim}{{\rm codim}}
\newcommand{\cont}{{\rm cont}}
\newcommand{\Cov}{\textup{Cov}}
\newcommand{\Crys}{{\rm Crys}}
\newcommand{\cts}{\textup{cts}}
\newcommand{\Div}{{\rm Div}}
\newcommand{\Dmod}{{\rm Dmod}}
\newcommand{\ECov}{{\rm ECov}}
\newcommand{\ed}{{\rm ed}}
\newcommand{\Ess}{{\rm EFin}}
\renewcommand{\et}{\textup{\'et}}
\newcommand{\ev}{\textup{ev}}
\newcommand{\Fdiv}{{\rm Fdiv}}
\newcommand{\Fib}{{\rm Fib}}
\newcommand{\FSets}{{\rm FSets}}
\newcommand{\FtAff}{{\rm FtAff}}
\newcommand{\Gal}{{\rm Gal}}
\newcommand{\height}{\textup{ht}}
\newcommand{\Hom}{{\rm Hom}}
\newcommand{\iinf}{\textup{inf}}
\newcommand{\im}{{\rm im}}
\newcommand{\Ker}{{\rm Ker}}
\newcommand{\LL}{\textup{L}}
\newcommand{\Loc}{{\rm Loc}}
\newcommand{\Max}{{\rm Max \ }}
\newcommand{\MIC}{\mbox{MIC}}
\newcommand{\Min}{{\rm Min \ }}
\newcommand{\NN}{\textup{N}}
\newcommand{\Mod}{\text{\sf Mod}}
\newcommand{\Noohi}{\textup{Noohi}}
\newcommand{\perf}{\textup{perf}}
\newcommand{\pet}{{\textup{proét}}}
\newcommand{\Pic}{{\rm Pic}}
\newcommand{\Rep}{\text{\sf Rep}}
\newcommand{\Res}{{\rm Res}}
\newcommand{\rank}{{\rm rank}}
\newcommand{\red}{{\rm red}}
\newcommand{\Sets}{\textup{Sets}}
\newcommand{\Spf}{\textup{Spf}}
\newcommand{\spe}{\textup{sp}}
\newcommand{\str}{\textup{str}}
\newcommand{\strat}{{\rm Str}}
\newcommand{\sym}{\text{Sym}}
\newcommand{\tp}{{\rm top}}
\newcommand{\Tr}{{\rm Tr}}
\newcommand{\trace}{{\rm Tr}}
\newcommand{\vect}{\text{\sf vect}}
\newcommand{\Vect}{\text{\sf Vect}}




\setcounter{section}{0}
\maketitle

\begin{abstract} Given a Noetherian formal scheme $\hat X$
    over
    $\Spf(R)$, where $R$ is a complete \textup{DVR}, we first prove a theorem of
    meromorphic descent along a possibly infinite cover of
    $\hat{X}$. Using
    this we construct a
    specialization functor from the category of continuous
    representations of the pro-étale fundamental group of the
    special fiber to the category of $F$-divided sheaves on the
    generic fiber. This specialization functor is compatible with the
    specialization functor of the étale fundamental groups. We also express the pro-étale
    fundamental group of a connected scheme $X$ of finite type over a field as
    coproducts and quotients of the free group and the étale fundamental groups of
    the normalizations of the irreducible components of $X$ and those
    of its singular loci.
\end{abstract}

\section*{Introduction}

Let \(X\) be a connected locally topologically notherian scheme.  
In \cite[Definition 7.4.2]{BS15}, B.~Bhatt and P.~Scholze introduced
the {\it pro-\'etale fundamental group} \(\pi_1^\pet(X)\) of \(X\), a
topological group
which classifies the {\it geometric covers} of \(X\) (cf.~\cite[Definition
7.3.1]{BS15}).  The geometric covers include the finite \'etale
covers.  Therefore, \(\pi_1^\pet(X)\) refines Grothendieck's \'etale
fundamental group \(\pi_1^\et(X)\), which classifies the finite
\'etale covers.  In fact, there is a natural morphism
\(\pi_1^\pet(X)\to \pi_1^\et(X)\) which makes \(\pi_1^\et(X)\) the
profinite completion of \(\pi_1^\pet(X)\) (in this introduction we
ignore the important issue of base points for simplicity of exposition).

We have \(\pi_1^\pet(X)=\pi_1^\et(X)\) when \(X\) is normal.  If \(X\)
is a degenerate curve (cf.~\cite{Mum72})  over an algebraically closed
field, then \(\pi_1^\pet(X)\) is a discrete free group (cf.~Theorem IV
below), a fact that naturally
occurred in the theory of Mumford curves (cf.~\cite{Mum72}), but could
only be phrased in an ad hoc way previously.

Let \(R\) be a complete discrete valuation ring (DVR) with
\({\rm Frac}(R)=K\), and let \(X\) be a proper scheme over \(\Spec R\).  
Let $X_0$ be the special fiber and $X_K$ the generic
fiber.  Grothendieck's {\it specialization map} 
\[
    {\rm sp}_{\et}\colon\pi_1^\et(X_K)\arr \pi_1^\et(X_0)
\]
has been a fundamental tool for studying the \'etale fundamental
groups.  But in general there is no analogous map \(\pi_1^\pet(X_K)\to
\pi_1^\pet(X_0)\), as pointed out by \cite[Remark 3.2]{Elena18}.  For
example, assume that the residue field of \(R\) is
algebraically closed and \(X_K\) is a Mumford curve, then any map
\(\pi_1^\pet(X_K)\to \pi_1^\pet(X_0)\), being a map from a profinite
group to a discrete free group, is necessarily trivial.  

E. Lavanda overcame this difficulty in a special case by constructing
a commutative diagram of \(\bar K\)-group schemes:
\[
  \xymatrix{
      \pi_1^{\rm Fdiv}(X_{\bar K}) \ar@{->}[rr]^{{\rm
      sp}_{\pet}}\ar@{->}[d]&&\pi_1^\pet(X_0)_{\bar K}\ar@{->}[d]\\
      \pi_1^\et(X_{\bar K})_{\bar K} \ar@{->}[rr]^{{\rm
      sp}_{\et}}&&\pi_1^\et(X_0)_{\bar K},
}
\]
where \({\rm sp}_\pet\) is a new specailization map constructed by
Lavanda, \(\pi_1^{\rm Fdiv}(X)\) is a Tannakian fundamental group
classifying \(F\)-divided sheaves, and for a topological group \(\pi\),
\(\pi_{\bar K}\) is the algebraic hull of \(\pi\) over \(\bar K\)
(cf.~\nameref{s:notation} (2), (4)).
For a profinite \(\pi\), \(\pi_{\bar K}\) is the rather transparent formation
of the constant group scheme associated to \(\pi\) (cf.~\nameref{s:notation}
(3)). 
Thus the above
commutative diagrams means that Lavanda's specialization map lifts Grothendieck's.  The vertical arrow on the left is a
functorial surjection (cf.~\cite[Proposition 13,
\S2.4]{dS2007}, \cite[Theorem I]{TZ1}).

Lavanda's hypothesis for the above theory is that \(X\) is a
semi-stable curve over $R$ whose generic fiber is smooth.  This is
quite restrictive.   The first main result in this paper is to do the
theory in much greater generality (another
approach of generalizing the specialization map by considering  the de Jong fundamental group of the
rigid generic fiber can be found in \cite{ALY21}):

\begin{thmI}[cf.~\ref{Specialization of the pro-etale
    fundamental group when R is strictly Henselian}]\label{first main result} Let $R$ be  a complete discrete
    valuation ring of equal characteristic $p>0$, and
    let $X$ be a proper scheme over $R$ with
    connected fibers \(X_K\) and \(X_0\). Assume further that the residue
    field of \(R\) is separably closed and \(X(R)\neq\emptyset\).  Then there is a commutative
    diagram
\[
  \xymatrix{
      \pi_1^{\rm Fdiv}(X_{K}) \ar@{->}[rr]^{{\rm sp}_{\pet}}\ar@{->}[d]&&\pi_1^\pet(X_0)_{K}\ar@{->}[d]\\
      \pi_1^\et(X_{K})_{K} \ar@{->}[rr]^{{\rm sp}_{\et}}&&\pi_1^\et(X_0)_{K},
}
\]
\end{thmI}

We remark that for a more general $R$-scheme $X$ the hypothesis \(X(R)\neq\emptyset\) is applicable upon
replacing \(R\) by a finite extension (cf.~\ref{Specialization of the etale fundamental
    group when R' is finite over R}).  Furthermore, the above
  commutative diagram comes from a refinement which is a
  \(2\)-commutative diagram of Tannakian categories, and this refined
  result is valid without assuming either \(X(R)\neq\emptyset\) or that
  \(R\) has a separably closed residue field (cf.~\ref{new construction
  recovers the old one}).

\bigbreak
There are two main ingredients in E. Lavanda's construction of $\spe_\pet$:
\begin{enumerate}[label=(\Alph*)]
    \item A formula for the pro-étale fundamental group of a semi-stable curve
        \cite[Theorem 1.17]{Elena18}; \item The theory of meromorphic descent
along a torsor under a torsion-free discrete group on a curve, developed by D. Gieseker in 
\cite[Lemma 1]{Gi73}. \end{enumerate}

Her construction of $\spe_\pet$ works only
for semi-stable curves
because of both (A) and (B).  In this paper, we generalize (B)
considerably by the following theorem:

\begin{thmII}[cf.~\ref{meromorphic descent for non Noetherian stuff}]\label{meromorphic descent for non Noetherian stuff0} If $q\colon \hat{Y}\to\hat{X}$ is a $G$-torsor, where $G$ is any discrete
    group, and if $\hat{X}$ is Noetherian, then the pullback
    along $q$ induces an equivalence between
    $\Coh(\hat{X})\otimes_RK$ and the category of sheaves in
    $\Coh(\hat{Y})\otimes_RK$ equipped with meromorphic descent
    data.
\end{thmII}

Thanks to Theorem~\ref{meromorphic descent for non Noetherian stuff0} we can
directly prove Theorem~\ref{first main result}  without using anything like (A).
On the other hand, we also generalize (A) to much more general schemes,
even though it is not needed in Theorem~\ref{first main result}. We
consider the case when $X$ is a connected Nagata
Noetherian J-2
scheme (e.g.~a connected scheme of finite type over a field). The idea is to express the pro-étale fundamental group
of $X$
in terms of the étale fundamental groups of the normalization of
its irreducible components and the pro-étale fundamental group
of its singular locus. Then using Noetherian induction one may
understand the pro-étale fundamental group based on one's 
knowledge on the étale fundamental groups.

\begin{thmIII}[cf.~\ref{structure theorem when the singular locus is
    connected}]\label{structure theorem in Introduction} Let $X$ be a connected Nagata Noetherian J-2
    scheme. Assume that $X$ is irreducible and the singular locus $Z$ of the reduced
    induced structure of $X$
    is connected.  Let \(\tilde X\) be the normalization of \(X\), and
    let $\{Z_j\}_{1\leq j\leq n}$ be the connected components of the
inverse image of $Z$ under $\tilde{X}\to X$.
   Then \(\pi_1^\pet(X)\) is the coproduct of
    \(X_1^\et(\tilde X)\), \(X_1^\et(Z,x)\), and a free discrete group of rank
    \(n-1\) on generators \(x_2,\ldots,x_n\), quotient by the relations 
\[
\phi_j(a)=x_j\psi_j(a)x_j^{-1},\quad  a \in \pi_1^\pet(Z_j), 1 \leq j\leq n,
\]
where \(\phi_j:\pi_1^\pet(Z_j)\to \pi_1^\pet(Z)\) and
\(\psi_j:\pi_1^\pet(Z_j)\to \pi_1^\pet(\tilde X)\) are functorial
maps, and \(x_1=1\).
\end{thmIII}

We have a similar description when \(X\) has more irreducible
components (cf.~\ref{structure theorem when the singular locus is
    connected}).  If \(Z\) has more connected components, we also
  have an inductive procedure to compute \(\pi_1^\pet(X)\) (cf.~\ref{structure theorem when singular locus is
disconnected}).  As an example, E. Lavanda's formula (\cite[Theorem
1.17]{Elena18}) is a special case of the following:

\begin{thmIV}[cf.~\ref{curve case}]
If $X$ is a connected scheme of finite type over a separably closed
field $k$ and if the singular locus $Z$ is 0-dimensional,
then \(\pi_1^\pet(X)\) is the coproduct of \(\pi_1^\et(X_i)\), \(1\leq
i\leq n\), and a free discrete group of rank \(|f^{-1}(Z)|-|Z|-n+1\),
where \(X_1,\ldots,X_n\) are the irreducible components of \(X\),
\(f:\tilde X\to X\) is the normalization map, and $|\cdot|$ denote the number of points of the
underlying set.
\end{thmIV}

As another application of Theorem \ref{structure theorem in
Introduction} we prove that a continuous representation of
$\pi_1^\pet(X,x)$ factors through a discrete quotient iff its
restrictions to the étale fundamental groups of the
normalizations of the irreducible components of $X$  do so (cf.~\ref{when the proetale representation
    factors through a discrete one}).

\section*{\centering Notations and Conventions}
\label{s:notation}

\begin{enumerate}
\item For a topological group \(\pi\) and a field $K$,
    $\Rep_K^\cts(\pi)$ denotes the category of finite dimensional continuous
    $K$-representations viewing $K$ as a discrete field (even if $K$ carries
    a natural non-discrete topology).

\item For a topological group \(\pi\) and a field $K$, we denote by
    \(\pi_{K}\coloneq \pi^\cts\) the \emph{algebraic hull} of \(\pi\) over \(K\)
(cf.~\cite[Definition 2.2]{Elena18}). This is the Tannaka dual of the
Tannakian category $\Rep_K^\cts(\pi)$ equipped with the forgetful fiber
functor. 
\item If $\pi$ is  a profinite group, say $\pi=\varprojlim_{i\in I}\pi_i$ where each
$\pi_i$ is a finite discrete group and the limit is taken in the category of topological
groups,
then $\pi_K=\varprojlim_{i\in I}(\pi_i)_K$, where the projective limit is
taken in the category of affine $K$-group schemes. Since 
the association $\pi\mapsto \pi_K$ induces an equivalence between profinite groups and
profinite constant $K$-group schemes, we will identify
$\pi$ and $\pi_K$ without saying.
\item Let $X$ be a scheme over a field $K$ of positive
    characteristic. We denote by $F_{X/K}\colon X\to X^{(1)}$ the
    \emph{relative Frobenius map}. For each $i\in \N$, we can define,
    inductively, the relative Frobenius $F_{X^{(i)}/K}\colon X^{(i)}\to
    X^{(i+1)}$.   An \emph{$F$-divided sheaf} is a sequence $(E_i, \sigma_i
)_{i\in\N}
$ where $E_i$ is a finitely presented $\sO_{X^{(i)}}$-module and
$\sigma_i\colon F_{X^{(i)}/K}^*E_{i+1}\to E_{i}$ is $\sO_{X^{(i)}}$-linear
isomorphism. We denote by $\Fdiv(X)$ the category of 
$F$-divided sheaves on $X/K$. It's worth to note that if $X$ is of finite type
over $K$, then $E_i$ is actually a vector bundle on $X^{(i)}$ 
(see \cite[\S 2.2.1]{dS2007}, \cite[Theorem I, (3)]{TZ1} for more details).
\end{enumerate}

\section{Meromorphic descent}
\subsection{Noetherian meromorphic descent} 

\begin{conv} Let $R$ be an adic Noetherian ring, and let $I\subseteq R$
be an
ideal of definition. Suppose that $I=(f)$. Let $K$ be a
localization of $R$ at the element $f$. Let $\hat{X}$ be an
\textit{adic
formal scheme} (\cite[10.4.2, p.~407]{EGA1new}) equipped with an \textit{adic morphism}
$\hat{X}\to\Spf(R)$ (\cite[10.12.1, p.~436]{EGA1new}). In this
case, we simply call $\hat{X}$ a formal
scheme over $\Spf(R)$ or over $R$. We denote by $\Coh(\hat{X})$ the category of sheaves of
    $\sO_{\hat{X}}$-modules on $\hat{X}$ whose pullbacks to the
    scheme
    $X_n\coloneq (\hat{X},\sO_{\hat X}/I^{n+1})$ are
    finitely presented for each $n\in\N$. If $\hat{X}$ is a locally Noetherian
    formal scheme (\cite[10.4.2, p.~407]{EGA1new}), then
    $\Coh(\hat{X})$ is just the category of coherent sheaves on
    $\hat{X}$
(\cite[\href{https://stacks.math.columbia.edu/tag/01XZ}{01XZ}]{stacks-project}).
\end{conv}

\begin{defn}
    We denote by $\Coh^m(\hat{X})$ the category  whose objects are
    exactly the same as those in $\Coh(\hat{X})$, whose set of morphisms between
    two objects $\sF,\sG$ is 
    \[
        \Gamma(\hat{X},\underline{\Hom}_{\sO_{\hat{X}}}(\sF,\sG)\otimes_RK)
    \]
    where $\underline{\Hom}$ denotes the internal sheaf Hom.
    \end{defn}

\begin{defn} Let $\sC$ be an $R$-linear category. We define the
    category $\sC\otimes_RK$ to the category whose objects are
    exactly the same as those in $\sC$, whose set of morphisms between
    two objects $X,Y$ is 
    $\Hom_\sC(X,Y)\otimes_RK$. If $\sC$ is an Abelian category,
    then so is $\sC\otimes_RK$ (\cite[Lemma 1.5, p.
    7]{DTZ18}).
\end{defn}

\begin{defn} \label{def of meromorphic descent for quasi-compact
    stuff}   
  Let $q\colon\hat{Y}\to \hat{X}$ be an adic map of locally  Noetherian formal
  schemes over $\Spf(R)$. We define the category
  $\Coh^m(\hat{Y}/\hat{X})$ of coherent sheaves with \textit{meromorphic
  descent data} (MDD for short) to be the category whose objects are pairs
  $(\sF,\phi)$, where $\sF$ belongs to $\Coh^m(\hat{Y})$ and $\phi\colon p_1^*\sF\arr p_2^*\sF$
  is a map  in
  $\Coh^m(\hat{Y}\times_{\hat{X}}\hat{Y})$  which
  satisfies 
  \begin{itemize}
      \item the identity condition: the pullback of $\phi$ along
          the diagonal map 
          $\hat{Y}\to \hat{Y}\times_{\hat{X}}\hat{Y}$ is the
          identity map in $\Coh^m(\hat{Y})$;
      \item the cocycle condition: we have an equality
          \[p_{23}^*\phi\circ p_{12}^*\phi=p_{13}^*\phi\] in
  $\Coh^m(\hat{Y}\times_{\hat{X}}\hat{Y}\times_{\hat{X}}\hat{Y})$,  \end{itemize}
  whose morphisms between two objects $(\sF,\phi)$ and
  $(\sG,\varphi)$ are those morphisms in
  \[\Hom_{\Coh^m(\hat{Y})}(\sF,\sG)\] which are
  compactifiable with $\phi$ and $\varphi$.
\end{defn}

\begin{thm} \label{meromorphic descent for quasi-compact stuff}
      Let $q\colon\hat{Y}\to \hat{X}$ be an adic map of Noetherian formal
  schemes over $\Spf(R)$. Then there is a canonical functor
  \[\Coh(\hat{X})\otimes_RK\arr\Coh^m(\hat{Y}/\hat{X})\] which
  is an equivalence if $q$ is faithfully flat.
\end{thm}
\begin{proof}  The functor is simply the formation of the {\it
    canonical descent datum}
    (\cite[\href{https://stacks.math.columbia.edu/tag/023D}{023D}]{stacks-project}).
    A proof of  this theorem,
    which is due to O. Gabber, can
    be found in \cite[(1.9), p.~774]{Ogus84} (the result was stated
    there when \(R\) is a complete DVR and \(K={\rm Frac}(R)\), but the
    proof works in our setting).
\end{proof}

\begin{ex}\label{meromorphic descent for non-Noetherian 
    stuff fails1} Theorem \ref{meromorphic descent for
quasi-compact stuff} would not be true if one replaces ``Noetherian'' by
    ``locally Noetherian''. 
 Let  $R$ be a complete \textup{DVR} with
 maximal ideal $I=(\pi)$ and quotient field $K$. Let $T$ (e.g.
 $R[x,y]/(xy)$) be a
    scheme over $R$ equipped with two disjoint opens $U_x,U_y$
    (e.g. $U_x=(x\neq 0)$, $U_y=(y\neq 0)$)   whose special
    fibers are not empty and $U_x\simeq U_y$ as $R$-schemes. Let
    $\{T_n|n\in\Z\}$ be $\Z$-copies of $T$. We glue $U_{y_n}\subseteq T_n$ with
    $U_{x_{n+1}}\subseteq T_{n+1}$ for all $n\in\Z$ in the sense of
    \cite[Chapter II, Exercise 2.12, p.~80]{Hartshorne77}. Then we get
    a scheme $X$ with a Zariski-covering \{$T_i\}_{i\in\Z}$
    such that
    \[
        T_i\bigcap T_j=
        \begin{cases}
            \emptyset, &\text{ if } i-j\neq\pm1;\\
            T_i=T_j, &\text{ if } i=j;\\
            U_{y_i}=U_{x_{j}},  &\text{ if } j=i+1;
        \end{cases}
    \]
    Set $Y=\coprod_{n\in\Z}T_n$, and let $q\colon
    \hat{Y}\to\hat{X}$ be the formal completion of the map $Y\to
    X$ along the special fiber. Now if we take $\sO_{\hat{Y}}$
    on ${\hat{Y}}$ 
    and glue the restrictions on $\hat{U}_{x_i}$ by the uniformizer $\pi$.
    The so obtained meromorphic descent data can not
    descent to $\hat{X}$.

This shows that meromorphic descent may not be effective even for a
Zariski covering with by infinitely many Zariski open sets.   So the next result is
somewhat surprising. 
\end{ex}

\subsection{Meromorphic descent along $G$-torsors}

\begin{conv}\label{convention for the descent along G-torsors} Let $R$ is a complete \textup{DVR}, and let $k$ be the
    residue field of $R$. Let $K$ be the quotient field. Let
    $\pi$ be a uniformizer
    of $R$, $S\coloneq \Spec(R)$, $\hat{S}\coloneq \Spf(R)$. Let
    $\hat X\to \hat{S}$ be an adic morphism of Noetherian formal
    schemes. We denote $X_n$ the scheme given by the ringed
space $(\hat X,\sO_{\hat X}/(\pi)^{n+1})$. Let
$q\colon \sYrho\to\hat{X}$ be a torsor under a discrete group
$G$, i.e. $q$ is a faithfully flat adic map of formal schemes
over $\hat{S}$ and there is a group homomorphism $\rho\colon G\to
\Aut_{\hat X}(\sYrho)$ such that the induced map
\[\sYrho\times G\arr \sYrho\times_{\hat X}\sYrho\]
\[(y,w)\longmapsto (y,\rho(w)(y))\] is an
isomorphism, where $\sYrho\times G\coloneq\coprod_{w\in
G}\sYrho$. Note that  $\sYrho$ is a locally
Noetherian formal scheme, but not Noetherian when \(G\) is infinite. We denote by $Y_n$ the scheme given by the ringed
space $(\sYrho,\sO_{\sYrho}/(\pi)^{n+1})$.\end{conv}

By using the isomorphism \(\sYrho \times _{\hatX}\sYrho\simeq \sYrho
\times G\), it is evident that \(\Coh^m(\sYrho/\hatX)\) is equivalent
to the category \(\Coh^m(\sYrho/\hatX,G)\), which we define now:

\begin{defn}\label{tensor theoretic MDD}  The category \(\Coh^m(\sYrho/\hatX,G)\) of {\it coherent
    sheaves with MDD} is the category whose objects are pair
  \((\sF,\{h_w\}_{w \in G})\) where \(\sF\) is an object in
  \(\Coh^m(\sYrho)\) and \(\{h_w\}_{w \in W}\) is a
  collection of elements such that
    \[h_w \in\Hom_{\Coh^m(\sYrho)}(\sF,w^*\sF)=
    \Gamma(\sYrho,\Homsh_{\sO_{\sYrho}}(\sF,w^*\sF)\otimes_RK)\]
satisfying the ``identity condition'' and the ``cocycle
condition'':
\begin{itemize}
      \item the identity condition: $h_e=\id$ for the unit $e\in G$;
      \item the cocycle condition: we have an equality
  \[w^*(h_{w'})\circ h_w=h_{w'w}\] for all $w,w'\in G$.  \end{itemize}
A map \[(\sF,\{h_w\}_{w\in G})\arr
(\sF',\{h_w'\}_{w\in G})\] between two coherent sheaves with MDD is an element in
\[\Hom_{\Coh^m(\sYrho)}(\sF, \sF')=\Gamma(\sYrho,\Homsh_{\sO_{\sYrho}}(\sF,\sF')\otimes_RK)\]
which is compatible with $\{h_w\}_{w\in G}$ and $\{h_w'\}_{w\in
G}$ in the obvious sense.\end{defn}

\begin{defn}  Let $(\sF,\{h_w\}_{w \in G}$ be a
    coherent sheaf with  MDD on $\sYrho$. We say
    that it is \textit{genuine} if for all \(w \in G\), the element
    \(h_w\) lies in the subset \(\Hom_{\Coh(\sYrho)}(\sF, w^*\sF)\)
    of   \(\Hom_{\Coh^m(\sYrho)}(\sF, w^*\sF)\). \end{defn}

%

\begin{lem}\label{meromorphic descent data are genuine} Any object in
  \(\Coh^m(\sYrho/\hatX,G)\) is isomorphic to a genuine one.
\end{lem}

\begin{proof}

Suppose that $(\sF,\{h_w\}_{w \in G})$ is an object in \(\Coh^m(\sYrho/\hatX,G)\). Let $\sG$ be the
image of $\sF$ in $\sF\otimes_RK$. Then we have
$\sF\otimes_RK\simeq \sG\otimes_RK$. Thus replacing $\sF$ by
$\sG$ we may assume that $\sF$ is $R$-torsion free. Now we claim that
there is a coherent subsheaf $\sF'\subseteq \sF\otimes_RK$ such that
$\sF'\otimes_RK=\sF\otimes_RK$ and  $h_w$ sends $\sF'\subseteq
\sF\otimes_RK$ to
$w^*\sF'\subseteq w^*\sF'\otimes_RK=w^*\sF\otimes_RK$.  Evidently, the
lemma follows from this claim.

We will find  $\sF'$ using Noetherian induction on \(\hatX\).  The key is the
induction step: if the claim is true when \(\sYrho/\hatX\) is replaced
by \(q^{-1}(U)/U\) for some open \(U\) of \(\hatX\), and \(U \neq \hatX\),
then there exists an open \(U'\) such that \(U \subsetneq U' \subset
\hatX\) and the claim is true when \(\sYrho/\hat X\) is replaced by \(q^{-1}(U')/U'\).

\def\injto{\hookrightarrow}
Let \(Z=\hat X\setminus U\).  We regard \(Z\) as a closed subscheme of
\(X_0\) with the reduced induced structure.  Let \(\eta\in Z\) be a generic
point of \(Z\).  Since
\[\pi_1^\pet(\Spec(\kappa(\eta)))=\pi_1^\et(\Spec(\kappa(\eta)))\] is profinite, the \(G\)-torsor \(\sYrho|_{\Spec(\kappa(\eta))}\) obtained by pulling
back \(\sYrho\) through 
\[
    \Spec(\kappa(\eta))\injto Z\injto X_0 \injto \hatX
\]
is induced from an \(H\)-torsor  over
\(\Spec(\kappa(\eta))\) with \(H\subseteq G\) finite. Spreading it out we
find $V\subseteq Z$, an open neighborhood of $x$, and an $H$-torsor $\tilde{Z}\to
V$ which induces $\hat{Y}|_{V}$.  Thus \(\tilde Z\) is an open in
\(\sYrho|_V=q^{-1}(V)\) and \(\sYrho|_V\) is the
disjoint union of \(w\tilde Z\) over \(w \in G/H\).

We can find an \(H\)-invariant  quasi-compact open subset \(W\) of \(Y_0\)
such that \(W \cap q^{-1}(Z)=\tilde Z\).  Consider \(W\) as an open
formal subscheme of \(\sYrho\).  We can convert the 
MDD on \(\sF\) in the sense of \ref{tensor theoretic MDD} to an MDD \(\phi\) in the sense of
\ref{def of meromorphic descent for quasi-compact stuff}, and restrict it to get an MDD for
\(W \to q(W)\).  By Gabber's theorem \ref{meromorphic descent for quasi-compact
stuff}, there is a coherent sheaf
\(\sE\) on
\(q(W)\) such that \((q|_W)^*\sE\) is isomorphic to \((\sF|_W,\phi)\) in
\(\Coh^m(W/q(W))\).  We may and do assume that the isomorphism is given by
a morphism in \(\Hom_{\Coh(W)}((q|_W)^*\sE, \sF|_W)\).  Let \(\sE '\) be the
image of this morphism.

Let \(\sF '\) be the coherent sheaf on \(q^{-1}(U)\) asserted by the induction
hypothesis.  We now use \(\sF '\) on \(q^{-1}(U)\) and \(\sE '\) on \(W\) to
form an sheaf \(\sF ''\) on \(q^{-1}(U')\) extending \(\sF '\), where \(U':=U\cup
q(W)=U\cup q(\tilde Z)\), using a
method of Gieseker and Raynaud, as follows.  We first try to extend it
to \(q^{-1}(U) \cup W=q^{-1}(U)\cup\tilde Z\).  On the open set
\(q^{-1}(U)\cap W\), the restriction of \(\sF, \sF ', \sE '\) become the same upon
tensoring with \(K\).  Therefore, for \(n\) large enough we have
\[
\pi^n\sE ' \subset \sF ' \subset \pi^{-n} \sE ' \quad \mbox{ on }q^{-1}(U)\cap W.
\]
This implies
\[
  \sF '/\pi^n\sE ' \subset \pi^{-n} \sE '/\pi^n \sE' \quad \mbox{ on
  }q^{-1}(U)\cap W.
\]
On the other hand, the right-hand side is obviously the restriction of
a coherent sheaf on \(W_{2n}\) (by which we mean \(W\) regarded as an
open subscheme of \(Y_{2n}\)).  By
\cite[\href{https://stacks.math.columbia.edu/tag/01PD}{01PD}]{stacks-project}, the
left-hand side extends to a coherent sheaf \(\sG\) on \(W\) such that \(\sG\) is a
subsheaf of \(\pi^{-n} \sE '/\pi^n \sE '\).  Let \(\sE ''\) be
the subsheaf of \(\pi^{-n}\sE '\) such that \(\sE ''/\pi^n\sE
'=\sG\).  By construction, \(\sE''=\sF '\) on \(q^{-1}(U)\cap W\).
Since both $\sE',\sF'$ are $H$-invariant and the extension in
\cite[\href{https://stacks.math.columbia.edu/tag/01PD}{01PD}]{stacks-project}
is functorial, $\sE''$ is also $H$-invariant.

We now define a sheaf \(\sF ''\) on
\[
    q^{-1}(U')=q^{-1}(U)\cup\bigcup_{w\in G} w^{-1}W
\]
as follows: it is
the subsheaf of \(\sF \otimes_R K\) generated by sections of \(\sF '\)
and sections of \(h_w^{-1}(w^*\sE '')\) (a sheaf on \(w^{-1}W\)) for all \(w \in
G\). Thus a local section of $\sF\otimes_RK$ is in $\sF''$ iff its stalk at
each point is a sum of stalks of some local  sections of $\sF'$ and $h_w^{-1}(w^*\sE '')$.  We claim that
\(\sF ''\) fulfills all the conditions requied by the induction step.

From the above definition and the induction hypothesis, it is easy to see that \(\sF
''_y=\sF '_y\) for all \(y \in q^{-1}(U)\).  For \(y
\in q^{-1}(U')\setminus q^{-1}(U)\), \(\sF ''_y=\sum h_w^{-1}(\sE
''_z)\), where the sum is over \(\{(w,z) : w \in G, z \in \tilde Z,
w^{-1}z=y\}\), a transitive \(H\)-set.  It follows that for any \(w
\in G\), 
\[
    \sF''|_{w^{-1}W}=\sum_{j\in
    H}h^{-1}_{jw}((jw)^*\sE'')=h_w^{-1}(w^*\sE'')
\]
where the second equality follows from the fact that both $W$ and $\sE''$ are
$H$-invariant. We conclude that \(\sF ''\) is
coherent on \(q^{-1}(U')\).

It remains to verify that \(h_u\) maps \(\sF''\) to \(u^*\sF''\) for each
$u\in G$.
It suffices to check this over \(q^{-1}(U)\) and \(w^{-1}W\) for all
\(w \in G\).  The case of \(q^{-1}(U)\) is clear.  Over \(w^{-1}W\),
we have the following commutative diagrams:
\[
    \begin{tikzpicture}[xscale=2.9,yscale=-1.2]
        \node (A0_1) at (1, 0) {$uw^{-1}W$};
        \node (A1_0) at (0, 2) {$w^{-1}W$};
        \node (A1_2) at (2, 2) {$W$};

                \path (A1_0) edge [->] node[auto]
                    {$\scriptstyle{u}$}
            (A0_1);
        \path (A1_0) edge [->] node[auto] {$\scriptstyle{w}$} (A1_2);
        \path (A0_1) edge [->] node[auto] {$\scriptstyle{wu^{-1}}$} (A1_2);
    \end{tikzpicture}
    \hspace{10pt}
  \begin{tikzpicture}[xscale=2.9,yscale=-1.2]
        \node (A0_1) at (1, 0) {$u^*h_{wu^{-1}}^{-1}((wu^{-1})^*\sE'')$};
        \node (A1_0) at (0, 2) {$h_w^{-1}(w^*\sE'')$};
        \node (A1_2) at (2, 2) {$w^*\sE''$};

                \path (A1_0) edge [->] node[auto]
                    {$\scriptstyle{h_u}$}
            (A0_1);
        \path (A1_0) edge [->] node[auto] {$\scriptstyle{h_w}$} (A1_2);
    \path (A0_1) edge [->] node[auto] {$\scriptstyle{u^*h_{wu^{-1}}}$} (A1_2);
    \end{tikzpicture}
\]
due to the cocycle condition $h_w=(u^*h_{wu^{-1}})\circ h_u$. 
Thus \(h_u\) maps \(\sF''|_{w^{-1}W}\) to
\(u^*(\sF''|_{uw^{-1}W})\) which is exactly $(u^*\sF'')|_{w^{-1}W}$.
 The lemma is proved completely.
\end{proof}

\begin{thm}\label{meromorphic descent for non Noetherian stuff}
    The pullback
    functor \[\Coh(\hat X)\otimes_RK\longrightarrow \Coh^m(\hat{Y}/\hat{X},G)\]
    is  an
    equivalence.
\end{thm}

\begin{proof}
It is enough to show that the functor is fully faithful. Then
the statement follows readily from \ref{meromorphic descent
data are genuine}. Now suppose we are given two $R$-torsion free sheaves
$\sF,\sG\in\Coh(\hat{X})$. Set $\Coh^m(q,G)\coloneq\Coh^m(\hat{Y}/\hat{X},G)$.

By fpqc-descent of
quasi-coherent sheaves, we have an exact
sequence of  $R$-torsion free modules:
\[0\arr\Hom_{\sO_{\hat{X}}}(\sF,\sG)\arr\Hom_{\sO_{\sYrho}}(q^*\sF,q^*\sG)\arr\prod_{w\in
G}\Hom_{\sO_{\sYrho}}(q^*\sF,q^*\sG)\]
where the second map is the pullback, and the third map is
$w^*-\id^*$.  Note that the natural map \[(\prod_{w\in
    G}\Hom_{\sO_{\sYrho}}(q^*\sF,q^*\sG))\otimes_RK\arr\prod_{w\in
G}(\Hom_{\sO_{\sYrho}}(q^*\sF,q^*\sG)\otimes_RK)\] is injective, so
we get an exact sequence:
\begin{equation}\label{sequence after descent}0\arr\Hom_{\sO_{\hat{X}}}(\sF,\sG)\otimes_RK\arr\Hom_{\sO_{\sYrho}}(q^*\sF,q^*\sG)\otimes_RK\arr\prod_{w\in
G}(\Hom_{\sO_{\sYrho}}(q^*\sF,q^*\sG)\otimes_RK)\end{equation} By the very
definition of the category $\Coh^m(q,G)$, we have an exact sequence 
\begin{equation}\label{sequence before descent}
\begin{aligned}0\arr\Hom_{\Coh^m(q,G)}(q^*\sF,q^*G)&\arr
    \Gamma(\sYrho,\Homsh_{\sO_{\sYrho}}(q^*\sF,q^*\sG)\otimes_RK)\\ &\arr\prod_{w\in
G}\Gamma(\sYrho,\Homsh_{\sO_{\sYrho}}(q^*\sF,q^*\sG)\otimes_RK)
    \end{aligned} \end{equation}
where the last map is again given by $w^*-\id^*$. The inclusion
\[\Hom_{\sO_{\sYrho}}(q^*\sF,q^*\sG)\otimes_RK \subseteq
\Gamma(\sYrho,\Homsh_{\sO_{\sYrho}}(q^*\sF,q^*\sG)\otimes_RK)\] which
identifies $\Hom_{\sO_{\sYrho}}(q^*\sF,q^*\sG)\otimes_RK$
as the  subspace
 of elements $s$ of the right hand side such that $\pi^ns$ is contained in
\[\Hom_{\sO_{\sYrho}}(q^*\sF,q^*G)\subseteq
\Gamma(\sYrho,\Homsh_{\sO_{\sYrho}}(q^*\sF,q^*\sG)\otimes_RK)\] for some
$n\in\N$, induces a map (\refeq{sequence after descent}) $\Rightarrow$
(\refeq{sequence before descent}). We have to show that the kernel part of (\refeq{sequence after descent}) $\Rightarrow$
(\refeq{sequence before descent}) namely, the map 
\[\Hom_{\sO_{\hat{X}}}(\sF,\sG)\otimes_RK\arr\Hom_{\Coh^m(q,G)}(q^*\sF,q^*G)
\]
is an isomorphism

To show this, we just
have to show that if
\[s\in\Gamma(\sYrho,\Homsh_{\sO_{\sYrho}}(q^*\sF,q^*\sG)\otimes_RK)\]
such that $w^*s=s$ for all $w\in G$, then there exist $n$ large such that
$\pi^ns\in\Hom_{\sO_{\sYrho}}(q^*\sF,q^*G)$.

Let $U\to X_0$ be a quasi-compact \'etale map which trivializes
the $G$-torsor $Y_0\to X_0$. Then the image $V$ of $U\to Y_0$ is a
quasi-compact open whose $G$-translates $\{V_w\}_{w\in G}$ cover
$Y_0$. Since $V$ is quasi-compact,  $\pi^ns|_V\in\Hom_{\sO_{{\sYrho}|_V}}(q^*\sF|_V, q^*\sG|_V)$.
Now the condition $w^*s=s$ implies that $\pi^ns|_{V_w}\in\Hom_{\sO_{{\sYrho}|_{V_w}}}(q^*\sF|_{V_w},
q^*\sG|_{V_w})$ for all $w\in G$. Thus $\pi^ns\in\Hom_{\sO_{\sYrho}}(q^*\sF,q^*G)$.
\end{proof}

\begin{rmk} \label{functoriality of meromorphic descent} The
    above pullback functor is functorial in the following
    sense. If there is a Cartesian diagram of locally Noetherian
    formal schemes over $\Spf(R)$
    \[
    \begin{tikzpicture}[xscale=2.0,yscale=1.2,bmr/.pic={\draw (0,0)--++(-90:2.5mm)--++(180:2.5mm)--++(90:2.5mm)--++(0:2.5mm);}]
        \path
            (0,0)     node (F) {$\sYrho'$}
            +(-25:.9) pic[scale=1,red]{bmr}
            +(0:1.5)  node (star) {$\sYrho$}
            ++(-90:1) node (X) {$\hat{X}'$}
            +(0:1.5)  node (Y) {$\hat{X}$};
        \draw[->] (F)--(star)
            node[midway,above,scale=.6]{$\beta$};
        \draw[->] (F)--(X) node[midway,anchor=east,scale=.6]{$q'$};
        \draw[->] (X)--(Y) node[midway,above,scale=.6]{$\alpha$};
        \draw[->] (star)--(Y) node[midway,anchor=west,scale=.6]{$q$};
    \end{tikzpicture} 
\] where $q,q'$ are $G$-torsors, 
and if $\sF$ is a coherent sheaf on $\hat{X}$ and
$(q^*\sF,\{h_w\}_{w\in G})$ is the corresponding pair, then the
pair corresponding to
$\alpha^*\sF$ is $(\beta^*q^*\sF,\{\beta^*h_w\}_{w\in G})$.
\end{rmk}

\subsection{The specialization functor}
We are in Setting \ref{convention for the descent along G-torsors}
except that here we assume that $R$ is of characteristic
$p$.
Let $X$ be a proper scheme over $\Spec(R)$, and let $\hat{X}$ be
    the formal scheme associated with the special fiber of
    $X\to \Spec(R)$. Suppose that the special fiber $X_0$ is
    connected and $\xi\in X_0$ is a geometric point. Set
    $X_K\coloneq X\times_{\Spec(R)}\Spec(K)$.

    \begin{thm}\label{existence theorem}
    There is a
    natural equivalence of categories:
    \[\Coh(\hat{X})\otimes_RK\xrightarrow{}\Coh(X_K)\] 
    which is additive, monoidal and exact.
\end{thm}
\begin{proof} By Grothendieck's existence theorem we have an
    additive, monoidal and exact equivalence
    \[\Coh(X)\otimes_RK\xrightarrow{\
    \simeq\ }\Coh(\hat{X})\otimes_RK\]One the other hand, since
    $X$ is Noetherian the natural pullback functor
    \[\Coh(X)\otimes_RK\arr\Coh(X_K)\]
    is an equivalence. This completes the proof.
\end{proof}

\begin{const}\label{construction of the specialization map for
    proetale} Let $\hat{X}^{(i)}$ denote the formal scheme
    obtained by the following Cartesian diagram \[
    \begin{tikzpicture}[xscale=2.0,yscale=1.2,bmr/.pic={\draw (0,0)--++(-90:2.5mm)--++(180:2.5mm)--++(90:2.5mm)--++(0:2.5mm);}]
        \path
            (0,0)     node (F) {$\hat{X}^{(1)}$}
            +(-28:.8) pic[scale=1,red]{bmr}
            +(0:1.5)  node (star) {$\hat{X}$}
            ++(-90:1) node (X) {$\Spf(R)$}
            +(0:1.5)  node (Y) {$\Spf(R)$};
        \draw[->] (F)--(star);
        \draw[->] (F)--(X);
        \draw[->] (X)--(Y)
            node[midway,below,scale=.6]{$F_{\Spf(R)}$};\draw[->] (star)--(Y);
        \end{tikzpicture} \] where $F_{\Spf(R)}$ denotes the
        absolute Frobenius of $\Spf(R)$. The universal property
        of pullback diagrams provides a map
        $\hat{X}\to\hat{X}^{(1)}$ over $\Spf(R)$. In other words,
        $\hat{X}\to\hat{X}^{(1)}$ is the \textit{relative
        Frobenius} of
        $\hat{X}/\Spf(R)$. Then we can define the relative
        Frobenius maps
        $F_{\hat{X}^{(i)}/R}\colon\hat{X}^{(i)}\to \hat{X}^{(i+1)}$ of
        $\hat{X}^{(i)}/\Spf(R)$ inductively for each $i\in\N$.
        Similarly, we have relative Frobenius maps
        $F_{X_K^{(i)}/K}\colon X_K^{(i)}\to X_K^{(i+1)}$
        of $X_K^{(i)}/K$.

        We want to construct the specialization functor
        \[\spe^K_{\pet}\colon
\Rep_K^\cts(\pi_1^\pet(X_0,\xi))\arr\Fdiv(X_K)\]
(cf.~\nameref{s:notation} (1), (4)).

Given a continuous
        representation \[\rho\colon\pi_1^\pet(X_0,\xi)\to
        \GL(V)\]
where $V$ is a finite dimensional $K$-vector space, we take $G\subseteq
\GL(V)$ the image of $\rho$. By \cite[Lemma 7.4.6]{BS15}, we get a $G$-torsor
$q_0\colon Y_0\to X_0$. Extending $q_0$ to higher thickenings we get a
$G$-torsor $q\colon \hat{Y}\to\hat{X}$.

Now $V\otimes_R\sO_{\hat{Y}}$ equipped with the $G$-action is an object in
$\Coh^m(\hat{Y}/\hat{X})$, so it corresponds, via \ref{meromorphic descent for
non Noetherian stuff} and \ref{existence theorem}, to a coherent sheaf $E_0$
on $X_K$. Moreover, we have commutative diagrams 
        \[
            \begin{tikzpicture}[xscale=2.0,yscale=1.2,bmr/.pic={\draw (0,0)--++(-90:2.5mm)--++(180:2.5mm)--++(90:2.5mm)--++(0:2.5mm);}]
                \path
                    (0,0)     node (F) {$\sYrho^{(i)}$}
                    +(-28:.8) pic[scale=1,red]{bmr}
                    +(0:1.5)  node (star) {$\sYrho^{(i+1)}$}
                    ++(-90:1) node (X) {$\hat{X}^{(i)}$}
                    +(0:1.5)  node (Y) {$\hat{X}^{(i+1)}$};
                \draw[->] (F)--(star);
                \draw[->]
                    (F)--(X)node[midway,left,scale=.6]{$q^{(i)}$};
                \draw[->] (X)--(Y)
                    node[midway,above,scale=.6]{};\draw[->]
                        (star)--(Y)node[midway,right,scale=.6]{$q^{(i+1)}$};
            \end{tikzpicture} 
        \] which are Cartesian because $\sYrho^{(i)}\to
        \hat{X}^{(i)}$ are \'etale. Each $q^{(i)}$ will produce a coherent
        sheaf $E_{i}$, and \ref{functoriality of meromorphic descent}
        guarantees the isomorphism $\sigma_i$.  
 In this way, we get an object $(E_i,\sigma_i)_{i\in\N}\in\Fdiv(X_K)$. 

 Note that instead of taking $G$ to be the image of $\rho$ we can also take
 any discrete quotient \[\pi_1^\pet(X_0,\xi)\twoheadrightarrow G\to
 \GL(V)\] in the middle of $\rho$. The object
 $(E_i,\sigma_i)_{i\in\N}\in\Fdiv(X_K)$ obtained in this way will be
 canonically isomorphic to the one we constructed above. Using this we see that
 our construction is indeed functorial, so it defines the desired
 functor $\spe^K_{\pet}$.
\end{const}

\begin{const}\label{construction of the specialization map for
    etale fundamental group}

    Let's consider the construction of the functor 
\[\spe^K_\et\colon
\Rep_K^\cts(\pi_1^\et(X_0,\xi))\arr\Fdiv(X_K)\]
which is a folklore.

    Given a finite
    dimensional continuous $K$-representation
    $\rho\colon\pi_1^\et(X_0,\xi)\to \GL(V)$, let
    $G\subseteq\GL(V)$ be the image of $\rho$. Then $G$ is a finite group and
    the surjection \[\pi_1^\et(X_0,\xi)\twoheadrightarrow G\]
    corresponds to a pointed $G$-torsor $q_0\colon Y_0\to X_0$ which extends
    automatically to a $G$-torsor $q\colon \hat{Y}\to \hat{X}$.
    By Grothendieck's existence theorem we get a (unique) $G$-torsor
$\mathfrak{q}\colon Y\to X$ whose formal completion along $X_0\hookrightarrow
X$ is $q$.
The general fiber of $\mathfrak{q}$ is a $G$-torsor $q_K\colon Y_K\to X_K$ which
extends to cartesian diagrams:
\begin{equation}\label{cartesian diagram of the relative
    Frobenius of the generic fibers}
            \begin{tikzpicture}[xscale=2.0,yscale=1.2, baseline=(current  bounding  box.center),bmr/.pic={\draw
                (0,0)--++(-90:2.5mm)--++(180:2.5mm)--++(90:2.5mm)--++(0:2.5mm);}]
                \path
                    (0,0)     node (F) {$Y_K^{(i)}$}
                    +(-28:.9) pic[scale=1,red]{bmr}
                    +(0:1.5)  node (star) {$Y_K^{(i+1)}$}
                    ++(-90:1) node (X) {${X}_K^{(i)}$}
                    +(0:1.5)  node (Y) {${X}_K^{(i+1)}$};
                \draw[->] (F)--(star);
                \draw[->]
                    (F)--(X)node[midway,left,scale=.6]{$q_K^{(i)}$};
                \draw[->] (X)--(Y)
                    node[midway,above,scale=.6]{};\draw[->]
                        (star)--(Y)node[midway,right,scale=.6]{$q_K^{(i+1)}$};
            \end{tikzpicture} 
\end{equation}
Now we apply fpqc-descent of quasi-coherent sheaves to descent the
        $G$-sheaf $\sO_{Y_K^{(i)}}\otimes_KV$ to a coherent sheaf $E_i$ over
        $X_K^{(i)}$ along the $G$-torsor $q^{(i)}$. We also have isomorphisms
        $\sigma_i\colon F_{X_K^{(i)}/K}^*E_{i+1}\to E_i$ given by
        \eqref{cartesian diagram of the relative Frobenius of the generic
        fibers} and the functoriality of the fpqc-descent of quasi-coherent
        sheaves. This defines an object $(E_i,\sigma_i)_{i\in\N}\in\Fdiv(X_K)$.
        Just as  \ref{construction of the specialization map for
        proetale}, the construction is functorial, so it defines  the
    functor \[\spe^K_\et\colon
\Rep_K^\cts(\pi_1^\et(X_0,\xi))\arr\Fdiv(X_K).\] 
%
\end{const}

\begin{thm} \label{new construction recovers the old one} The following natural diagram of categories
    \[\begin{tikzpicture}[xscale=4.9,yscale=-1.2]
        \node (A0_0) at (0, 0) {$\Rep_K^\cts(\pi_1^\et(X_0,\xi))$};
        \node (A0_1) at (1, 0) {$\Fdiv(X_K)$};
        \node (A1_0) at (0, 1) {$\Rep_K^\cts(\pi_1^\pet(X_0,\xi))$};
        \node (A1_1) at (1, 1) {$\Fdiv(X_K)$};

        \draw[>=latex,->] (A0_0) --
            (A0_1)node[midway,above,scale=.9]{$\spe^K_\et$};

        \draw[>=latex,->] (A1_0) -- 
            (A1_1)node[midway,above,scale=.9]{$\spe^K_{\pet}$};

        \path (A0_1) edge [->] node[auto]
            {$\scriptstyle{=}$} (A1_1);
        \path (A0_0) edge [->] node[auto]
            {$\scriptstyle{}$} (A1_0);
    \end{tikzpicture}
\]
is 2-commutative.
\end{thm}

\begin{proof}
    The category $\Rep_K^\cts(\pi_1^\et(X_0,\xi))$ is the full
    subcategory of $\Rep_K^\cts(\pi_1^\pet(X_0,\xi))$ consisting
    of representations whose monodromy groups are finite. Suppose in
    \ref{construction of the specialization map for proetale} the image
    $G\subseteq \GL(V)$ is finite. Then the $G$-torsor $q\colon \hat{Y}\to\hat{X}$ is algebraizable,
    so there exists a unique  $G$-torsor $\mathfrak{q}\colon Y\to X$ whose
    formal completion is $q$. To finish the proof one just has to notice that if
    we identify $\Coh(\hat{X})\otimes_RK$ (resp. $\Coh(\hat{Y})\otimes_RK$) with  
    $\Coh(X_K)$ (resp. $\Coh(Y_K)$) via \ref{existence theorem}, then the
    meromorphic descent in
    \ref{meromorphic descent for non Noetherian stuff} is nothing but the
    fpqc-descent. 
\end{proof}

\subsection{The specialization map of fundamental groups} We resume the
notations and conventions in the previous subsection. We assume, in addition,
that  $X$ and $X_K$ are connected, and that $\eta\in X_K$ is a geometric point.
These additional assumptions are just to
make sense of the fundamental groups of $X$ and $X_K$.  

Recall that we have a specialization map of the \'etale
    fundamental groups
    \begin{equation}\label{eet specialization - article body}
    \pi_1^\et(X_K,\eta)\arr
    \pi_1^\et(X,\eta)\simeq
    \pi_1^\et(X,\xi)\xleftarrow{\ \ \cong\ \
    }\pi_1^\et(X_0,\xi)
\end{equation} 
which is defined up to a choice of a
    path from $\eta$ to $\xi$. We want to recover this map from the
    specialization functor $\spe^K_\et$ (cf.~\ref{construction of the
    specialization map for etale fundamental group}), which does not depend on $\eta$. 

Suppose $\eta$ comes from a rational point $\eta_0$,
\emph{i.e.} there is a factorization
\[\eta\colon\Spec(\bar{K})\to\Spec(K)\xrightarrow{\eta_0}
X_K\] The pro-constant quotient $\pi^G(X_K/K,\eta_0)$ of the
Nori-étale fundamental group $\pi^E(X_K/K,\eta_0)$ (cf.~\cite[Definition 4.5]{Zhang2015} or the gerbe version \cite[Definition 4.1]{TZ1}) is the quotient of the étale fundamental group
$\pi_1^\et(X_K,\eta)=\pi^G(X_K/K,\eta)$
classifying finite étale covers equipped with a $K$-\textit{rational
point} over $\eta_0$. We have fully faithful inclusions:
\[
    \Rep_K^\cts(\pi^G(X_K/K,\eta_0))\subseteq\Rep_K(\pi^E(X_K/K,\eta_0))\subseteq\Fdiv(X_K)
\]
where the first inclusion corresponds to the pro-constant quotient map, and the second inclusion is obtained by taking the
essentially finite objects of $\Fdiv(X_K)$ (cf.~\cite[Corollary 12, Proposition 13]{dS2007},\cite[Theorem 5.8, Theorem 6.23]{TZ1}). The category $\Fdiv(X_K)$ is Tannakian
(see \cite[Theorem I (3)]{TZ1}
). The rational point $\eta_0$ provides a neutral fiber functor
\[\ev_{\eta_0}\colon \Fdiv(X_K) \arr \Vect_K\]
\[
    \hspace{20pt}    (E_i,\sigma_i)_{i\in\N}\longmapsto \eta_0^*E_0
\]    for
$\Fdiv(X_K)$. We denote
$\pi_1^\Fdiv(X_K,\eta_0)$ the affine $K$-group scheme corresponding
to the neutral Tannakian category $(\Fdiv(X_K),\ev_{\eta_0})$.

\subsubsection{When $X$ has an  $R$-rational point}

\begin{prop} \label{Specialization of the etale fundamental group
    when R is strictly Henselian} Suppose that the residue field
    $k$ of $R$ is separably closed, and
        there is a $R$-rational point $(\xi,\eta_0)\in X(R)$. Then the functor $\spe^K_\et$ factors as 
\[
    \Rep_K^\cts(\pi_1^\et(X_0,\xi))\arr\Rep_K^\cts(\pi^G(X_K/K,\eta_0))\subseteq\Fdiv(X_K)
\]
 which induces 
maps of neutral Tannakian categories 
\begin{equation}\label{specialization of neutral Tannakian categories} (\Rep_K^\cts(\pi_1^\et(X_0,\xi)),
    F_{\xi})\arr(\Rep_K^\cts(\pi^G(X_K/K,\eta_0),F_{\eta_0})\arr
    (\Fdiv(X_K),\ev_{\eta_0})
\end{equation}
where $F_\xi$ and $F_{\eta_0}$ denote the forgetful functor to $\Vect_K$.
By taking Tannakian dual we get  homomorphisms of $K$-group
schemes
\begin{equation}
    \pi_1^\Fdiv(X_K,\eta)\relbar\joinrel\twoheadrightarrow\pi^G(X_K/K,\eta_0)\arr\pi_1^\et(X_0,\xi).\end{equation}
    By composing the second map with the quotient map $\pi_1^\et(X_K,\eta)\twoheadrightarrow
\pi^G(X_K/K,\eta_0)$ one recovers the specialization  
 map \eqref{eet specialization - article body}
    \[
\pi_1^\et(X_K,\eta)\twoheadrightarrow
\pi^G(X_K/K,\eta_0)\longrightarrow \pi_1^\et(X_0,{\xi}).
\]
Moreover, there is a canonical path between the geometric points $\xi$ and
\[\eta\colon \Spec(\bar{K})\arr \Spec(K)\xrightarrow{\eta_0} X.\]
\end{prop}
\begin{proof} Let's revisit \ref{construction of the specialization map for
    etale fundamental group}. Consider the pointed $G$-torsor $q_0$ corresponding to $\pi_1^\et(X_0,\xi)\to G$
\[
    \begin{tikzpicture}[xscale=2.9,yscale=-1.2]
        \node (A0_0) at (0, 0) {$$};
        \node (A0_1) at (1, 0) {$Y_0$};
        \node (A1_0) at (0, 1) {$\Spec(k)$};
        \node (A1_1) at (1, 1) {$X_0$};

        \draw[>=latex,->] (A1_0) -- node[auto]{$\scriptstyle{\xi}$}
            (A1_1);

        \path (A0_1) edge [->] node[right]
            {$\scriptstyle{q_0}$} (A1_1);
        \path (A1_0) edge [->] node[auto] {$\scriptstyle{y_0}$}
            (A0_1);
    \end{tikzpicture}
\]
Since $R$ is strictly
Henselian, all the finite étale covers split completely. Thus the geometric
point $y_0\in Y_0\subseteq Y$ determines a unique $R$-rational point
$(y_0,y_1)$ of $Y$.
\[
    \begin{tikzpicture}[xscale=2.9,yscale=-1.2]
        \node (A0_0) at (0, 0) {$$};
        \node (A0_1) at (1, 0) {$Y$};
        \node (A1_0) at (0, 1) {$\Spec(R)$};
        \node (A1_1) at (1, 1) {$X$};

        \draw[>=latex,->] (A1_0) --
            node[auto]{$\scriptstyle{(\xi,\eta_0)}$}
            (A1_1);

        \path (A0_1) edge [->] node[right]
            {$\scriptstyle{\mathfrak{q}}$} (A1_1);
        \path (A1_0) edge [dashed,->] node[auto] {$\scriptstyle{(y_0,y_1)}$}
            (A0_1);
    \end{tikzpicture}
\]
So we get the pointed torsor 
\begin{equation} \label{generic fiber}
    \begin{tikzpicture}[xscale=2.9,yscale=-1.2,baseline=(current  bounding  box.center)]
        \node (A0_0) at (0, 0) {$$};
        \node (A0_1) at (1, 0) {$Y_K$};
        \node (A1_0) at (0, 1) {$\Spec(K)$};
        \node (A1_1) at (1, 1) {$X_K$};

        \draw[>=latex,->] (A1_0) -- node[auto]
            {$\scriptstyle{\eta_0}$}
            (A1_1);

        \path (A0_1) edge [->] node[right] {$\scriptstyle{q_K}$} (A1_1);
        \path (A1_0) edge [->] node[auto] {$\scriptstyle{y_1}$}
            (A0_1);
    \end{tikzpicture}
\end{equation}
Since the pointed torsor \eqref{generic fiber} is pointed by a $K$-rational
point, the map $\rho_K\colon\pi_1^\et(X_K,{\eta})\to G$ corresponding to
the it factors as 
\[
\pi_1^\et(X_K,{\eta})\twoheadrightarrow
\pi^G(X_K/K,\eta_0)\longrightarrow G
\]
Continuing the algorithms in \ref{construction of the specialization map for
etale fundamental group} we get  the desired factorization of $\spe^K_\et$.

To check 
\eqref{specialization of neutral Tannakian categories} one first notice that
$q_K^*E_0=\sO_{Y_K}\otimes_KV$. Thus we have
\[
    F_\xi(\rho)=F_{\eta_0}(\rho_K)=V=y_1^*(\sO_{Y_K}\otimes_KV)=y_1^*q_K^*E_0=\eta_0^*E_0=\ev_{\eta_0}((E_i,\sigma_i)_{i\in\N}).
\]
Since all the identifications are functorial, we get \eqref{specialization of
neutral Tannakian categories}.

Finally, the canonical path $\mathfrak{q}^{-1}(\xi)\to \mathfrak{q}^{-1}(\eta)$ is given functorially by
\[
    y_0\longmapsto (y_0,y_1)\longmapsto
    (\bar{y}_1\colon\Spec(\bar{K})\to\Spec(K)\xrightarrow{y_1}Y_K\subseteq
Y)
\]
\end{proof}

\begin{cor}\label{Specialization of the pro-etale fundamental
group when R is strictly Henselian}
Suppose that the residue field
    $k$ of $R$ is separably closed, and
        there is an $R$-rational point $(\xi,\eta_0)\in X(R)$. Then the functor
        $\spe^K_\pet$ induces a specialization map 
        \begin{equation}\label{the specialization map of pro-etale rational
            point}
            \pi_1^\Fdiv(X_K,\eta_0)\arr(\pi_1^\pet(X_0,\xi))_K
\end{equation}
which recovers \eqref{eet specialization - article body}.
\end{cor}
\begin{proof} From \ref{construction of the specialization map for proetale}
    we get a $K$-linear tensorial exact functor $\spe^K_{\pet}$ of Tannakian categories. The $R$-rational
    point $(\xi,\eta_0)$ provides a $K$-rational point $\eta_0$ for $X_K$ and therefore a
    neutral fiber functor $\ev_{\eta_0}$ for $\Fdiv(X_K)$. To get \eqref{the
    specialization map of pro-etale rational point}, we just have to show
    that $\ev_{\eta_0}$ is compatible with the forgetful functor
    $F_\xi$ of $\Rep_K^\cts(\pi_1^\pet(X_0,\xi))$ under $\spe^K_{\pet}$.

    Let's revisit \ref{construction of the specialization map for proetale}.
    The $R$-rational point $(\xi,\eta_0)$ induces a $\Spf(R)$-rational
    point $\zeta$ of $\hat{X}$. Consider the pointed $G$-torsor $q_0$
    corresponding to $\pi_1^\pet(X_0,\xi)\to G$
    \[
    \begin{tikzpicture}[xscale=2.9,yscale=-1.2]
        \node (A0_0) at (0, 0) {$$};
        \node (A0_1) at (1, 0) {$Y_0$};
        \node (A1_0) at (0, 1) {$\Spec(k)$};
        \node (A1_1) at (1, 1) {$X_0$};

        \draw[>=latex,->] (A1_0) -- node[auto]{$\scriptstyle{\xi}$}
            (A1_1);

        \path (A0_1) edge [->] node[right]
            {$\scriptstyle{q_0}$} (A1_1);
        \path (A1_0) edge [->] node[auto] {$\scriptstyle{y_0}$}
            (A0_1);
    \end{tikzpicture}
\]
Since $R$ is strictly Henselian, there is a
    unique (hence functorial) lift $\lambda$ of $\zeta$ along $q$.
    \[
    \begin{tikzpicture}[xscale=2.9,yscale=-1.2]
        \node (A0_0) at (0, 0) {$$};
        \node (A0_1) at (1, 0) {$\hat{Y}$};
        \node (A1_0) at (0, 1) {$\Spf(R)$};
        \node (A1_1) at (1, 1) {$\hat{X}$};

        \draw[>=latex,->] (A1_0) --
            node[auto]{$\scriptstyle{\zeta}$}
            (A1_1);

        \path (A0_1) edge [->] node[right]
            {$\scriptstyle{{q}}$} (A1_1);
        \path (A1_0) edge [dashed, ->] node[above]
            {$\scriptstyle{\lambda}$}
            (A0_1);
    \end{tikzpicture}
\]
Suppose $\sE_0$ is the sheaf in $\Coh(\hat{X})\otimes_RK$ obtained by
applying meromorphic descent \ref{meromorphic descent for non Noetherian
stuff} to $V\otimes_R\sO_{\hat{Y}}$. Since $\sE_0$ corresponds to $E_0$ via
\ref{existence theorem}, $\zeta^*\sE_0\in
\Coh(\Spf(R))\otimes_RK$ corresponds to $\eta_0^*E_0\in \Vect_K$ by naturality
of \ref{existence theorem}.  But
\[\zeta^*\sE_0=\lambda^*(q^*\sE_0)=\lambda^*(V\otimes_R\sO_{\hat{Y}})=V\otimes_R\sO_{\Spf(R)}\]
corresponds to $V$ via \ref{existence theorem}. Thus
$\ev_{\eta_0}((E_i,\sigma_i)_{i\in \N})=\eta_0^*E_0=V$ as
desired.

The last statement is nothing but a combination of \ref{Specialization of the etale fundamental
group when R is strictly Henselian} and \ref{new construction recovers the old
one}.
\end{proof}

\subsubsection{When $X$ does not have an $R$-rational point}
We see that if $X$ admits an $R$-rational point, then $\spe^K_{\pet}$ induces a map of
fundamental group schemes \eqref{the specialization map of pro-etale rational
point} recovers the classical specialization map of fundamental groups
\eqref{eet specialization - article body}. What if $X$ has no $R$-rational point?

Suppose the residue field $k$ of $R$ is separably closed, and $X$
is flat over $R$. By \ref{lift the special point to
R'-point}
there is a finite ring extension $R\to R'$ where $R'$ is a complete \textup{DVR}
such that $X(R')\neq\emptyset$. Let $(\xi,\eta_0)\in X(R')$, and let $k'$
(resp. $K'$) denotes the residue field (resp. function field) of $R'$. Set
$X'\coloneq X\otimes_RR'$. Then $X_0'=X_0\otimes_kk'$ and
$X_K'=X_K\otimes_KK'$. Since
$k'/k$ is a finite purely inseparable extension,  we have 
\[
    \pi_1^\et(X_0,\xi)=\pi_1^\et(X_0\otimes_kk',\xi)\hspace{15pt}\text{and}\hspace{15pt}\pi_1^\pet(X_0,\xi)=\pi_1^\pet(X_0\otimes_kk',\xi).
\]
In this case, we have

\begin{cor}\label{Specialization of the etale fundamental group
    when R' is finite over R} Suppose the residue field
    $k$ of  $R$  is separably closed, and $X$ is flat and
    geometrically connected over $R$. Then a
homomorphism  of $K'$-group
schemes:
\[
    \pi_1^\Fdiv(X_K',\eta_0)\arr(\pi_1^\pet(X_0,\xi))_{K'}
\]
which recovers \eqref{eet specialization - article body} for $X'/R'$.
\end{cor}
\begin{proof} We apply \ref{Specialization of the pro-etale fundamental
group when R is strictly Henselian} to the $R'$-scheme $X'$.
\end{proof}

\begin{lem} \label{can choose a closed point} Let $A$ be a flat finitely
    generated $R$-algebra. Suppose that $A$ is an
    integral domain, and $P$ is a maximal ideal of $A$
    containing the uniformizer $\pi$ of $R$. If
    $A_P[\frac{1}{\pi}]$ is a field, then $A[\frac{1}{\pi}]$ is
    also a field.
\end{lem}

\begin{proof} The conditions imply that if $Q\subseteq P$ is a
    non-zero prime ideal of $A$, then $\pi\in Q$. Consider the
    minimal prime ideals of the ideal $(\pi)$ in $A_P$. If
    $PA_P$ was not one of them, then there would be $a\in PA_P$
    such that $a$ is not contained in any of those minimal prime
    ideals of $(\pi)$. By Krull principal ideal theorem each
    minimal prime ideal $I$ of $(a)$ has height 1, but since it is
    non-zero, $\pi\in I$. Thus $I$ has to be also a minimal
    prime ideal of $(\pi)$ - a contradiction! So $PA_P$ is a
    minimal prime ideal of $(\pi)$. This implies that
    $\height(P)=1$, so by \cite[14.C, Theorem 23, p.~84]{Mat-CA}
    and the fact that $R$ is universally catenary, 
    the function field of
    $A[\frac{1}{\pi}]$ is finite over $K$, hence $A[\frac{1}{\pi}]$
    a field.
\end{proof}

\begin{lem}\label{lift the special point to R'-point}
    If  $X$ is flat and locally of finite type over a complete \textup{DVR} $R$, then for any closed point
    $\xi\in X_0$ we
    can find a complete \textup{DVR} $R'$ which is finite flat over $R$ and
    an
    $R$-map $\Spec(R')\to X$ whose special point goes to $\xi$.
\end{lem}
\begin{proof} Replacing $X$ by an affine open neighborhood of $\xi$, we may assume
    that $X=\Spec(A)$ is affine and $\xi$ corresponds to a maximal ideal
    $P\subseteq A$. By going down,
    we can find a prime ideal $I\subseteq A$ such that $I\subseteq P$ and
    $I\bigcap R=\{0\}$. Thus the ring $A_P[\frac{1}{\pi}]$ is non-zero. Let $J\subseteq A[\frac{1}{\pi}]$ be a prime ideal
    such that $J_P\subseteq A_P[\frac{1}{\pi}]$ is maximal. By \ref{can choose
    a closed point}, $J$ is a maximal ideal of $A[\frac{1}{\pi}]$. Let
    $\eta\in X_K=\Spec(A[\frac{1}{\pi}])$ be the closed point
    corresponding to $J$, then the residue field
    $\kappa(\eta)$ is a finite extension of $K$. Note also that $\eta$
    specializes to $\xi\in X$.  By \cite[\href{https://stacks.math.columbia.edu/tag/054F}{054F}]{stacks-project} there exists a
    DVR $R'$ and a map $\Spec(R')\to X$, whose special point goes
    to $\xi$ and whose generic point goes to $\eta$. Moreover,
    $R'$ can be chosen in such a way that its function field is equal to
    $\kappa(\eta)$. Thus $R\to R'$ is a finite ring
extension
(cf.~\cite[\href{https://stacks.math.columbia.edu/tag/0335}{0335}]{stacks-project},
\cite[\href{https://stacks.math.columbia.edu/tag/03GH}{03GH}]{stacks-project}).\end{proof}

\bigskip

\section{The Pro-étale fundamental group} 

In this section, we generalize E. Lavanda's
theorem \cite[Theorem 1.17]{Elena18} to more general schemes (see
Theorem \ref{structure theorem when the singular
locus is connected}, Theorem \ref{structure theorem when singular locus
is disconnected} and Theorem \ref{curve case}).  The
core techniques we use here are similar to those in E. Lavanda's proof: (1)
the technique of proper descent of étale morphisms
\cite[Proposition 1.16]{Elena18} due to A. Grothendieck and  D. Rydh; (2)
a combinatorial method which turns out to be a folklore (see e.g.
\cite[8.4.1, p.~333]{Br06}). 

\subsection{Results about Noohi groups}

\begin{lem} \label{Small facts about Noohi groups} A Noohi group $\pi$ is Hausdorff and has a basis of
    open neighborhoods of $e\in \pi$ given by subgroups. 
\end{lem}
\begin{proof} Let \(\pi\)-\((\Sets)\) be the category of discrete sets
  with \(\pi\)-actions, and let $F_\pi\colon
    \pi\mbox{-}(\Sets)\to(\Sets)$ be the forgetful functor. Since $\pi$ is
    Noohi, the continuous map $\pi\to\Aut(F_\pi)$ is an
    isomorphism. Let $T$ denote the set of open subgroups of
    $\pi$. Then for each $U\subseteq \pi$ in $T$ the coset
    $\pi/U$ is a discrete set equipped with a continuous
    $\pi$-action. Let $\Aut(\pi/U)$ be the automorphism group (with the compact-open topology) of
    $\pi/U$ . Then
    $\Aut(F_\pi)$ is a subgroup of $\prod_{U\in T}\Aut(\pi/U)$
    equipped with the subspace topology. Since each
    $\Aut(\pi/U)$ is Hausdorff and has a basis of open
    neighborhoods of the unit given by subgroups, so is the product $\prod_{U\in T}\Aut(\pi/U)$ and its
    topological subgroup $\Aut(F_\pi)\cong\pi$.
\end{proof}

\begin{lem}\label{faithfulness of Tannakian}
Let $f,g\colon G\to \pi$ be two continuous maps of topological
groups, where $\pi$ is Hausdorff and has a basis of open
neighborhoods of $e\in\pi$. If the induced functors $f^*,g^*:\pi\mbox{-}(\Sets)\to G\mbox{-}(\Sets)$
coincide, then $f=g$.  In other words,
\(\Hom(G,\pi)\to\mathop{\rm Fun}(\pi\mbox{-}(\Sets),G\mbox{-}(\Sets))\) is injective.
\end{lem}
\begin{proof}
Since $\pi$ is Hausdorff and has a basis of open
neighborhoods of $e\in\pi$, the intersection of all open
subgroups of $\pi$ is $e$. Therefore, it is enough to show that
$f$ is equal to $g$ after composing with $\pi\to\pi/U$ for all
open subgroups $U\subseteq \pi$. But $\pi/U\in\pi\mbox{-}(\Sets)$, so
its restrictions to $G\mbox{-}(\Sets)$ via $f$ and $g$ must coincide. This
implies exactly that $\pi\to\pi/U$ equalizes $f$ and $g$.
\end{proof}

The following is a generalization of \cite[Lemma 2.51, p.~18]{Lara19} to not necessarily Hausdorff topological groups. 

\begin{lem}
    The natural inclusion 
    \[
        (\textup{Noohi groups}) \hspace{5pt}
        \subseteq\hspace{5pt} (\textup{Topological groups})
    \]
    has a left adjoint, which we will denote by $(-)^\Noohi$.
\end{lem}
\begin{proof}
    Given $G\in\textup{(Topological groups)}$ the category
    $G\mbox{-}(\Sets)$ is a tame
    infinite Galois category. Let $F\colon G\mbox{-}(\Sets)\to(\Sets)$
    be the forgetful functor, then $\Aut(F)$ equipped with the
    compact-open topology is a Noohi group by \cite[Theorem
    7.2.5.1]{BS15}. By \cite[Theorem
    7.2.5.3]{BS15} the
    continuous map $\alpha\colon G\to\Aut(F)$ identifies
    $\Aut(F)\mbox{-}(\Sets)$ and $G\mbox{-}(\Sets)$. We put
    \(G^{\Noohi}\coloneqq\Aut(F)\).  If $ \gamma\colon G\to G'$ is a
    continuous homomorphism, where $G'$ is Noohi, then by
    \cite[Theorem 7.2.5.2]{BS15} the
    functor
    $G'\mbox{-}(\Sets)\to G\mbox{-}(\Sets)$ corresponds to a unique map
    $\beta\colon \Aut(F)\to
    G'$. Thus by \ref{faithfulness of Tannakian} we have
    $\beta\circ\alpha=\gamma$. This concludes the proof.
\end{proof}

\begin{defn} \label{Noohi quotient}
Let $\pi$ be a Noohi group, and let $H\subseteq \pi$ be a
subgroup. Denote $\langle H\rangle$ the smallest normal subgroup
of $\pi$ containing $H$. We define the \textit{Noohi quotient of
  \(\pi\) by \(H\)} 
to be
$(\pi/\langle H\rangle)^\Noohi$. Let $\langle H\rangle^c$ be the
closure of $\langle H\rangle$ in $\pi$. By \cite[Proposition 7.1.5,
p.~188]{BS15} and \ref{Small facts about Noohi groups},
$(\pi/\langle H\rangle)^\Noohi$ is the Raĭkov
completion of $\pi/\langle H\rangle^c$.

If \(\{f_i,g_i\}_{i\in I}\) are elements of \(\pi\), we often form the
Noohi quotient \((\pi/\langle H\rangle)^\Noohi\) with \(H\) being the
subgroup generated by \(\{f_i^{-1}g_i\}_{i\in I}\).  We refer to this
as the {\it Noohi quotient of \(\pi\) by the relations
  \(\{f_i=g_i\}_{i\in I}\).}
\end{defn}

\begin{lem}\label{coproduct exists in Noohi groups} Finite coproducts
  exist in the category of Noohi groups.
\end{lem}

\begin{proof}
It suffices to show that $\pi_1\coprod^\pi \pi_2$ exists for any diagram
$\pi_1\xleftarrow{p}\pi\xrightarrow{q}\pi_2$ of Noohi groups.  Indeed
when \(\pi\) is trivial, $\pi_1\coprod \pi_2$ is just the Noohi
group corresponding to the tame infinite Galois category of discrete sets equipped
with two independent actions from $\pi_1$ and $\pi_2$. See
\cite[Example 7.2.6]{BS15}.

In general, $\pi_1\coprod^\pi \pi_2$ is nothing but the
    Noohi quotient of $\pi_1\coprod\pi_2$ by the relations
    $\{p(x)=q(x)\}_{x\in\pi}$. Alternatively, it can also be
    defined by the Noohi group corresponding to the tame infinite
    Galois category of discrete sets equipped
with two  actions from $\pi_1$ and $\pi_2$ which
agree on $\pi$.
\end{proof}


\subsection{A construction of Van Kampen}

The construction we are about to discuss, in the case of discrete
groups, can be found in \cite[8.4.1, p.~333]{Br06}, where it is attributed
to Van Kampen.

Let \(s\geq 1\) and let \(\pi,\pi ',\pi_1'',\ldots,\pi_s''\) be
Noohi groups.  Let \(\psi_i: \pi_i''\to \pi\), \(\phi_i:\pi_i''\to \pi
'\) be continuous homomorphisms.  We define a
(discrete) group \(F\) by generators  \(\{u_{ij}: 1\leq
i,j\leq s\}\) and relations \(u_{ii}=e\), \(u_{ij}u_{jk}=u_{ik}\).
Evidently, \(F\) is free of rank \(s-1\) on the generators
\(v_2,\ldots,v_s\), where \(v_i:=u_{1i}\) for \(i=2,\ldots,s\).  We
put \(v_1=e\).

\begin{lem}\label{vk-cons-1}
  The following Noohi groups are isomorphic:
  \begin{itemize}
    \item[{\rm(i)}] The Noohi coproduct of \(\pi '\) and \(F\).
    \item[{\rm (ii)}] The Noohi coproduct of \(s\) copies of \(\pi '\),
      and one copy of \(F\), then Noohi
      quotient by the relations \(u_{ij}^{-1}[y]_iu_{ij}=[y]_j\),
      \(1\leq i,j\leq s\), \(y \in \pi '\).  Here \([y]_i\) denote \(y\)
      regarded as an element in the \(i\)-th copy of \(\pi '\).
    \end{itemize}
  \end{lem}

  \begin{proof}  The isomorphism is characterized by \(y
    \leftrightarrow [y]_1\) for \(y \in \pi '\) and \(g
    \leftrightarrow g\) for \(g \in F\).  The universal property of
    coproduct implies that there are homomorphisms in both directions
    matching these elements, and they are inverse to each other.
  \end{proof}

  \begin{lem}\label{vk-cons-2}
    The following Noohi groups are isomorphic:
  \begin{itemize}
    \item[{\rm (i)}] The Noohi coproduct of \(\pi\) and the group in
      \ref{vk-cons-1} {\rm (i)}, Noohi quotient by the relations
      \(\psi_i(a)=v_i^{-1}\phi_i(a)v_i\), \(a \in \pi ''_i\),
      \(i=1,\ldots,s\).
  \item[{\rm(ii)}] The Noohi coproduct of \(\pi\) and the group
    in \ref{vk-cons-1} {\rm (ii)}, Noohi quotient by the relations
    \(\psi_i(a)=[\phi_i(a)]_i\), \(a \in \pi ''_i\), \(i=1,\ldots,s\).
    \item[{\rm (iii)}] The Noohi coproduct of \(F\) with
      \(\pi\coprod\limits^{\pi_1'',\phi_1,\psi_1}\pi '\),
      Noohi quotient by the relations
      \(\psi_i(a)=v_i^{-1}\phi_i(a)v_i\), \(a\in \pi ''_i\), \(i=2,\ldots,s\).
      \item[{\rm (iv)}] The Noohi coproduct of \(F\) and \((\)the Noohi fiber coproduct of
        \(\pi\to\pi\coprod\limits^{\pi_1'',\phi_1,\psi_1}\pi ',\ldots,\pi\to \pi\coprod\limits^{\pi_s'',\phi_s,\psi_s}\pi ')\), Noohi quotient by the
        relations \(u_{ij}^{-1}(e*_iy)u_{ij}=e*_jy\) for all
\(y \in \pi '\), \(i,j=1,\ldots,s\).  Here \(e*_i y\) denote
the product of \(e \in \pi\) and \(y\in \pi '\) in \(\pi\coprod\limits^{\pi_i'',\phi_i,\psi_i}\pi '\).
    \end{itemize}
\end{lem}

\begin{proof}  The isomorphism between (i) and (ii) is induced by that
  of \ref{vk-cons-1}.  The isomorphism between (i) and (iii) is rather
  obvious. The isomorphism between (iii) and (iv) is constructed in a
  way similar to that of \ref{vk-cons-1}.
\end{proof}

\begin{rmk} We will denote the group in \ref{vk-cons-2} as \({\bf
    VK}(\pi,\pi
  ';\pi_1'',\ldots,\pi_s'')_{(\psi_1,\ldots,\psi_s),(\phi_1,\ldots,\phi_s)}\)
  and omit the homomorphisms in subscripts when there is no
  confusion.  The description (i) is the one usually found in the
  literature, e.g.~\cite[8.4.1, p.~333]{Br06}.  It relies on single
  out the index \(1\) and so does (iii).  The descriptions (ii)
  and (iv) make it clear that the construction actually treats all
  indices in equal footing.  This construction is very useful due to
  the next result.
\end{rmk}

\def\scrC{\mathscr{C}}
\def\scrD{\mathscr{D}}
\def\scrE{\mathscr{E}}
\begin{sett}
  To state the result, we work in the following situation with
  \(s\geq 1\):
  \begin{itemize}
    \item Let \((\scrC,F)\), \((\scrD,G)\),
  \((\scrE_1,H_1),\ldots,(\scrE_s,H_s)\) be tame infinite Galois
  categories.
  \item For \(j=1,\ldots,s\), let
  \(u_j:\scrC\to \scrE_j\) (resp.~\(v_j:\scrD\to \scrE_j\)) be  a functor
 such that \((\scrC,H_j\circ u_j)\) (resp.~\((\scrD,H_j\circ v_j)\) is a tame infinite Galois
 category and \(H_j\circ u_j\) is compatible with \(F\)
 (resp.~\(H_j\circ v_j\) is compatible with \(G\)) in the sense of
 \cite[Definition 3.14]{Lara19}.
 \item Recall that an object of the 2-fibre product
     (\cite[\href{https:https://stacks.math.columbia.edu/tag/003R}{003R}]{stacks-project})
\[
\scrC \mathop{\times}\limits_{\scrE_1\times\cdots\times \scrE_s} \scrD
\]
is a triple \((X,Y,\phi)\), where \(X \in \scrC\), \(Y \in \scrD\),
and \(\phi\) is an isomorphism 
\[
    \bigl(u_j(X)\bigr)_{1\leq j\leq
s}\to\bigl(v_j(Y)\bigr)_{1\leq j\leq s}
\]
in $\scrE_1\times\cdots\times \scrE_s$. We define a functor \(F'\) from this
2-fibre product to \(({\rm Sets})\) by setting \(F'(X,Y,\phi)=F(X)\).
\item For \(j=1,\ldots,s\), choose an isomorphism of fiber functors on
  \(\scrC\colon H_j\circ u_j \xrightarrow{\cong} F\) and denote the
  composition \(\pi_1(\scrE_j,H_j)\to \pi_1(\scrC,H_j\circ u_j) \to
  \pi_1(\scrC,F)\) by \(\psi_j\).  We also choose an isomorphism of fiber functors on
  \(\scrD\colon H_j\circ v_j \xrightarrow{\cong} G\) and denote the
  composition \(\pi_1(\scrE_j,H_j)\to \pi_1(\scrD,H_j\circ v_j) \to
  \pi_1(\scrD,G)\) by \(\phi_j\).
\end{itemize}
\end{sett}

\begin{prop} \label{van kampen machine}
  With the above setting, the pair
\[
(\scrC \mathop{\times}\limits_{\scrE_1\times\cdots\times \scrE_s} \scrD,F')
\]
is a tame infinite Galois category, whose fundamental group is
isomorphic  to
\[
{\bf VK}(\pi_1(\scrC,F),\pi_1(\scrD,G);\pi_1(\scrE_1,H_1),\ldots,\pi_1(\scrE_s,H_s))_{(\psi_1,\ldots,\psi_s),(\phi_1,\ldots,\phi_s)}.
\]

\end{prop}

\begin{proof} Let us denote the pair in the proposition by \((\scrC
  ',F')\).  An object in \(\scrC '
  \) is an \((s+2)\)-tuple \((A,B,\lambda_1,\ldots,\lambda_s)\), where
  \(A \in \scrC\), \(B\in \scrD\), and
  \(\lambda_j\colon u_j(A)\xrightarrow{\cong} v_j(B)\).  We will view this
  as an object \((A,B,\lambda _1)\) of \(C\times_{\scrE_1} \scrD\),
  together with isomorphisms \((\lambda _2,\cdots,\lambda_s)\)
  in $\scrE_2\times\cdots\times \scrE_s$.

By \cite[Theorem 7.2.5]{BS15}, we may further describe a category equivalent to
\(\scrC '\) using Noohi group actions.  Indeed, \(\scrC '\) is
equivalent to the category of \(s\)-tuples
\((S,\lambda_2',\ldots,\lambda_s')\), where \(S\) is a discrete
\(\Pi_1\)-set with
\(\Pi_1:= \pi_1(\scrC,F)\mathop{\textstyle\coprod}\limits_{\pi_1(\scrE_1,H_1),\psi_1,\phi_1}\pi_1(\scrD,G)\),
and for \(s=2,\ldots,s\), \(\lambda_j'\colon S \to S\) is a bijection of
\(\pi_1(\scrE_j,H_j)\)-sets where the source (resp.~target) has the
structure of a \(\pi_1(\scrE_j,H_j)\)-set via \(u_j(A)\)
(resp.~\(v_j(B)\)).  This exactly means:
\(\phi_j(a).\lambda '_j(s)=\lambda '_j(\psi_j(a).s)\) for all \(a \in
\pi_1(\scrE_j,H_j)\), \(s \in S\).  Therefore, \(\scrC '\) is exactly
equivalently the category of discrete \(\pi\)-set, where \(\pi\) is
the Noohi group given by the proposition, by \ref{vk-cons-2} (iii).
This completes the proof by \cite[Example 7.2.2]{BS15}.  
\end{proof}

\begin{rmk} With the notation of the proof, it is obvious that under the identification
  \(\pi_1(\scrC ',F')\simeq \pi\), the maps \(\pi_1(\scrC ,F) \to
  \pi_1(\scrC ',F')\), \(\pi_1(\scrD,G)) \to \pi_1(\scrC ',F')\)
  correspond to the canonical maps \(\pi_1(\scrC,F)\to \pi\),
  \(\pi_1(\scrD,G)\to \pi\).
\end{rmk}

\begin{rmk}
More sophisticated forms of this sort of result can be found in
\cite[Corollary 5.3, p.~19]{stix2006} and \cite[Corollary 3.18,
p.~32]{Lara19}.  However, the form given here is more readily
applicable and is enough for almost all known applications, such as
\cite[Theorem~1.17]{Elena18}, which we generalize in
\ref{curve case}.  More applications are given in the next few subsections.
\end{rmk}

\subsection{Van Kampen theorems}

\begin{rmk}\label{pass to reduced structure} From now on we assume that $X$ is a locally topologically
    Noetherian scheme. Let $\Cov(X)$ denote the category of \textit{geometric covers of
    $X$} in the sense of \cite[Definition 7.3.1]{BS15}.  This
  construction gives a category fibered over the category of locally
  topological Noetherian schemes in an obivious way.
  
  Let $X_\red$
    denote the reduced induced scheme structure of $X$. Then by
    \cite[Exposé VIII, Théorème 1.1, p.~247]{SGA4} or \cite[Exposé IX, 4.10,
    p.~186]{SGA1} we can conclude that
    $\Cov(X)\xrightarrow{\simeq}\Cov(X_\red)$ just as in
    \cite[Lemma 1.15]{Elena18}. Thus if $X$ is a locally
    Noetherian connected scheme and $x\in X$ is a geometric point, then
    $\pi_1^\pet(X_\red,x)\to\pi_1^\pet(X,x)$ is an isomorphism.
\end{rmk}

\begin{lem}\label{general van kampen}
    Let $f_1:Z_1\rightarrow X, f_2: Z_2\to X$ be monomorphisms
    of
    schemes.  Assume that the induced map
    \(Z_1\coprod
Z_2\to X\) is a morphism of effective descent for  \(\Cov(-)\).  
  Then we have an equivalence given by the pullback functor:
  \[\Cov(X)\simeq \Cov(Z_1)\times_{\Cov(Z)}\Cov(Z_2)\]
  where $Z=  Z_1\times_XZ_2$. 
\end{lem}

\begin{proof} 
    By the effectiveness of descent, an
  object of \(\Cov(X)\) amounts to a descent datum, which consists of \(Y_i\in
  \Cov(Z_i)\), \(i=1,2\) together with \(\varphi_{ij}: p_1^* Y_i \to
  p_2^* Y_j\) for \(i,j \in \{1,2\}\) satisfying the cocycle
  condition.  The cocycle condition says that \(\varphi_{11}\) and \(\varphi_{22}\) are
  just identities (by the monomorphism assumption) and
  \(\varphi_{12},\varphi_{21}\) give
  isomorphisms \(Y_1|_Z \simeq Y_2|_Z\), inverse to each other.  Thus we get an object in the
  \(\Cov(Z_1) \times_{\Cov(Z)} \Cov(Z_2)\).  It is routine to verify
  that this construction gives an equivalence.
\end{proof}

\begin{lem} \label{van kampen}
  The asssumption of Lemma~\ref{general van kampen} is satisfied when
  \(Z_1, Z_2\) are open subschemes of \(X\) such that
  \(X=Z_1\cup Z_2\).  
\end{lem}

\begin{proof}  This is essentially the classical van Kampen theorem
  and an easy case of fpqc descent.  Notice that here
  \(Z=Z_1 \cap Z_2\).
\end{proof}

\begin{lem}\label{closed van kampen}
The asssumption of Lemma~\ref{general van kampen} is satisfied when
\(X\) is locally Notherian and
  \(Z_1, Z_2\) are closed subschemes of \(X\) such that
  \(X=Z_1\cup Z_2\) set-theoretically.
\end{lem}
\begin{proof}
    This follows from \cite[Proposition 1.16]{Elena18} applied
    to the map
    $\phi\colon Z_1\coprod Z_2\arr X$, which relies on Rydh's
    work \cite{Rydh10}
    and generalizes \cite[IX Th\'eor\`eme 4.12]{SGA1}.
\end{proof}

\begin{lem} \label{closed van kampen v2}
Let $X$ be a locally Noetherian scheme. Let $Z\subseteq X$ be a closed subscheme.
Consider a proper surjective map
\[\tilde{X}\xrightarrow{\hspace{5pt}f\hspace{5pt}}X\]
Denote $Z\times_X \tilde{X}$ by \(\tilde Z\). Suppose that the union
of  the image
of the two closed immersions
\[
  \Delta_f:\tilde X\to \tilde X\times_X\tilde X,\quad
  \tilde Z\times_Z \tilde Z \to \tilde X \times_X \times  \tilde X
\]
is \(\tilde X\times _X \tilde X\) set-theoretically.
Then the pullback functor induces an equivalence of categories:
\[\Cov(X)\xrightarrow{\cong} \Cov(\tilde{X})\times_{\Cov(\tilde
Z)}\Cov(Z).\]
\end{lem}

\begin{proof}
Let us construct a quasi-inverse of the pullback functor.  So we start with
a triple \((Y,W,\phi)\), where $Y\in\Cov(\tilde{X})$, $W\in\Cov(Z)$, and
$\phi: Y \times _{\tilde X}\tilde Z\xrightarrow{\cong} W\times_Z
\tilde Z$ is an isomorphism.
The idea is to construct a descent datum of \(Y\) for the morphism
\(f\), then we get the desired object in \(\Cov(X)\) by applying
\cite[Proposition 1.16]{Elena18}.

To construct this descent datum, we need to work on \(\tilde
X\times_X \tilde X\) using the equivalence
\[
\Cov(\tilde{X}\times_X\tilde{X})\cong
\Cov(\tilde X)\times_{\Cov(\tilde Z)}\Cov\left(
\tilde Z\times_Z\tilde Z \right),
\]
which follows from Lemma \ref{closed van kampen} because
the fiber product of the two closed immesions in the current lemma is the
closed immersion \(\tilde Z \xrightarrow{\Delta_f}\tilde
X\times_X\tilde X\).  Under this equivalence, \(p_i^*(Y)\) corresponds
to the triple \((Y,q_i^*(Y|_{\tilde Z}),{\rm can}_i)\) for \(i=1,2\).
Here \(p_i:\tilde X\times_X \tilde X\to \tilde X\), \(q_i:\tilde Z \times_Z\tilde
Z\to Z\) are projections, and \({\rm can}_i\) is the obvious canonical
isomorphism.  The desired descent datum \(\lambda:p_1^*(Y) \to p_2^*(Y)\)
corresponds to \((Y,q_1^*(Y|_{\tilde Z}),{\rm
  can}_1)\to(Y,q_2^*(Y|_{\tilde Z}),{\rm can}_2)\) given by \({\rm
  id}_Y: Y \to Y\), \(\varphi: q_1^*(Y|_{\tilde Z})\to
q_2^*(Y|_{\tilde Z})\), where \(\varphi\) is the canonical descent
datum signifying the fact that \(Y|_{\tilde Z}\) is isomorphic to \(W
\times_Z \tilde Z\) via \(\phi\).

We have to show that \(\lambda\) is indeed a descent datum.  That is,
considering the triple fibred product, the fibred product and the projections:
\[
\begin{tikzpicture}[xscale=5,yscale=2.2]
    \node (A0_0) at (0, 0) {$\tilde{X}\times_X\tilde{X}\times_X\tilde{X}$};
    \node (A0_1) at (1, 0) {$\tilde{X}\times_X\tilde{X}$};
    \node (A0_2) at (2, 0) {$\tilde{X}$};
    
    \draw[>=latex,->] ([yshift= 2pt] A0_0.east) to
        [out=40,in=140]   node[above,scale=0.5]{$p_{12}$}   ([yshift= 2pt] A0_1.west);
    \draw[>=latex,->] (A0_0.east) -- node[above,midway,scale=0.5]{$p_{23}$} 
        (A0_1.west);
    \draw[>=latex,->] ([yshift=-2pt] A0_0.east) to
        [out=-40,in=-140]    node[below,scale=0.5]{$p_{13}$}   ([yshift=-2pt] A0_1.west);

    \draw[>=latex,->]                ([yshift= 1pt] A0_1.east)
        to [out=30,in=150] node[above,scale=0.5]{$p_{1}$}([yshift= 1pt] A0_2.west);
    \draw[>=latex,->] ([yshift=-1pt] A0_1.east) to
       [out=-30,in=-150] node[below,scale=0.5]{$p_{2}$} ([yshift=-1pt] A0_2.west);
\end{tikzpicture},
\]
we have to show the cocycle condition $p_{23}^*\lambda\circ
p_{12}^*\lambda=p_{13}^*\lambda$. Since the fibred
product is covered by $\Delta_f(\tilde{X})$ and
$\tilde Z \times_Z \tilde Z$, the triple fibred product is
covered by the triple diagonal $\Delta_f^3(\tilde{X})$ and the
closed subset
\[\tilde Z\times_Z\tilde Z \times_Z \tilde Z
\] 
Thanks to
\cite[\href{https://stacks.math.columbia.edu/tag/0BTJ}{0BTJ}]{stacks-project} it is enough to
check the equality
$p_{23}^*\lambda\circ
p_{12}^*\lambda=p_{13}^*\lambda$ on $\Delta_f^3(\tilde{X})$ and
$\tilde Z \times _Z \tilde Z \times_Z \tilde Z$ separately. On
$\Delta_f^3(\tilde{X})$ all the pullbacks of $\lambda$ are
identities, so there is nothing to check. On 
$\tilde Z \times _Z \tilde Z \times_Z \tilde Z$ the equality holds
because $\lambda|_{\tilde Z\times_Z \tilde Z}=\varphi$  is a descent
datum for \(\tilde Z \to Z\).
\end{proof}

\subsection{Schemes with connected singularity}

\begin{conv}\label{nagata-J2} Let $X$ be a Nagata quasi-compact scheme, and let
    $\eta_1,\cdots,\eta_n$ be the generic points of $X$. Then
    we denote $f\colon \tilde{X}\arr X$ the relative
    normalization of the map 
    \[
        \coprod_{1\leq i\leq
        n}\Spec(\kappa(\eta_i))\longrightarrow X
    \]
    in the sense of
    \cite[\href{https://stacks.math.columbia.edu/tag/0BAK}{0BAK}]{stacks-project}.
    According to
    \cite[\href{https://stacks.math.columbia.edu/tag/0AVK}{0AVK}]{stacks-project},
    $f$ is finite (and surjective). Moreover, according to
    \cite[\href{https://stacks.math.columbia.edu/tag/035I}{035I}]{stacks-project}, the map $f\colon \tilde{X}\arr X$
    factors through the reduced induced structure
    $X_\red\subseteq X$, and the map $\tilde{X}\to X_\red$ is
    the normalization of $X_\red$ relative to the disjoint union
    of the spectrums of the generic points. See
    \cite[\href{https://stacks.math.columbia.edu/tag/0359}{0359}]{stacks-project}
    for examples of Nagata schemes.

    Assume further that $X$ is a J-2 scheme
    (\cite[\href{https://stacks.math.columbia.edu/tag/07R2}{07R2}]{stacks-project};
    see
    \cite[\href{https://stacks.math.columbia.edu/tag/07R5}{07R5}]{stacks-project}
    for examples of J-2 schemes).
    Then the singular locus $Z\subseteq X$ is closed. We give $Z$
     the reduced induced structure and put \(\tilde Z \coloneqq Z\times_X
     \tilde X\). Now it is easy to see that all the hypotheses of Lemma~\ref{closed van
      kampen v2} are satisfied.

    By
    \cite[\href{https://stacks.math.columbia.edu/tag/035A}{035A}]{stacks-project},
    \cite[\href{https://stacks.math.columbia.edu/tag/07R4}{07R4}]{stacks-project}
    and \ref{pass to reduced structure}, we may
    directly assume that
    $X$ is reduced in the proof of the following theorems.
\end{conv}

\begin{thm}\label{structure theorem when the singular
    locus is connected-covers}
    Let $X$ be a connected Nagata
    Noetherian J-2
     scheme.
Let $X_1,\cdots, X_n$ be its irreducible components. Let
$Z\subseteq X$ be the singular locus of $X_\red$.  
Let $\tilde{X}_i$ be the connected component of
$\tilde{X}$ corresponding to $X_i$, and put $\tilde Z\coloneqq
\tilde{X}\times_XZ$.  For each $1\leq i\leq n$, suppose  $\tilde{X}_i\bigcap
\tilde Z=\textstyle Z_{i1}\coprod \cdots\coprod Z_{in_i}$ is a
decomposition.  Then the pullback functor
\[\Cov(X)\xrightarrow{\ \simeq\ }\prod^{1\leq i\leq
n}_{\Cov(Z)}\left(\Cov(\tilde{X}_i)\mathop{{\times}}\limits_{\Cov(Z_{i1})\prod\cdots\prod\Cov(Z_{in_i})}\Cov(Z)\right)\]
 is an equivalence.
\end{thm}

\begin{proof}  By \ref{closed van kampen v2} the  pullback map

    \[\Cov(X)\xrightarrow{\ \simeq\ }\left(\Cov(\coprod_{1\leq i\leq
        n}\tilde{X}_i)\mathop{\times}\limits_{\Cov(\coprod_{ij}Z_{ij})}\Cov(Z)\right)=\left[\left(\prod_{1\leq i\leq
n}\Cov(\tilde{X}_i)\right)\mathop{\times}\limits_{\prod_{ij}\Cov(Z_{ij})}\Cov(Z)\right]\]is
an equivalence. Then one checks easily that the later is the
same as the right hand side in the statement of the theorem.
\end{proof}

Now we come back to fundamental groups. Assumptions and
notations are as above, and we assume in addition that $Z$ is
connected.  Choose geometric points $x\in Z$,
$x_{ij}\in  Z_{ij}$ and isomorphisms of fiber functors
$F_x\simeq F_{f(x_{ij})}$, hence isomorphisms
$\pi_1^\pet(Z,x)\simeq\pi_1^\pet(Z,f(x_{ij}))$. We denote
$\phi_{ij}$ the map
\[\pi_1^\pet(Z_{ij},x_{ij})\arr\pi_1^\pet(Z,f(x_{ij}))\xrightarrow{
\ \simeq\ }\pi_1^\pet(Z,x)\] induced by
$Z\subseteq X$, and $\psi_{ij}$ the map
\[
\pi_1^\pet(Z_{ij},x_{ij})\to
\pi_1^\pet(\tilde{X}_i,x_{ij})\xrightarrow{\ \simeq \ }\pi_1^\et(\tilde{X}_i,x_{i1})
\]
obtained by choosing isomorphisms of fiber functors $F_{x_{i1}}\simeq \cdots \simeq
F_{x_{in_i}}$, therefore isomorphisms
$\pi_1^\pet(\tilde{X}_i,x_{i1})\simeq\cdots\simeq
\pi_1^\pet(\tilde{X}_i,x_{in_i})$. Thanks to the equivalence
between the category of tame infinite Galois categories with fiber functors and
the category Noohi groups, we get

\begin{thm}\label{structure theorem when the singular
    locus is connected}
    There is an isomorphism of topological groups 
\[
\pi_1^{\textup{pro}\et}(X,x)\simeq\coprod_{1\leq i\leq
n}^{\pi_1^\pet(Z,x)}{\bf VK}\bigl(\pi_1^\et(\tilde{X}_i,x_{i1}),\pi_1^\pet(Z,x);\pi_1^\pet(Z_{i1},x_{i1}),\ldots,\pi_1^\pet(Z_{in_i},x_{in_i})\bigr)
\]

Moreover, the canonical maps $\pi_1^\pet(Z,x)\to\pi_1^\pet(X,x)$ and 
$\pi_1^\et(\tilde{X}_i,x_{i1})\to\pi_1^\pet(X,f(x_{i1}))\simeq\pi_1^\pet(X,x)$
correspond, via this isomorphism, to the projections in
the coproduct.
\end{thm} 

\begin{proof}
This is immediate from \ref{van kampen machine} and \ref{structure theorem when the singular
    locus is connected-covers}.
\end{proof}

\subsection{Disconnected singularities: a d\'evissage}

Next, we deal with the case when the singular locus is not
connected.

\begin{sett}\label{devissage} Let $X$ be a connected Nagata Noetherian J-2 scheme.
Let $X_1,\cdots, X_n$ be its irreducible components. Let
$Z_j\subseteq
X$ $(1\leq j\leq m)$ be the connected components of the singular
locus \(Z\) of $X_\red$. 
For \(j=1,\ldots,m\), set
\[T_j\coloneqq\left(\bigcup_ {X_i\cap Z_j\neq \emptyset}X_i\right)\setminus \left(\bigcup_{t\neq
    j}Z_t\right).
\]
We notice that if \(m=0\), then \(X\) is regular and \(n=1\).  From
now on we assume \(m \geq 1\).
\end{sett}

\begin{lem}\label{devissage lemma}
  With the above setting,
  \begin{itemize}
  \item[{\rm (i)}]  Each \(T_j\) is connected and open in \(X\).
  \item[{(ii)}] \(X=T_1\cup\cdots\cup T_m\).
  \item[(iii)] For each \(1\leq j\leq m\), \(Z_j \subset T_j\) and
    \(Z_j\) is disjoint from \(T_t\) for \(t \neq j\).
     \item[{\rm (iv)}] There exists \(k\) such that \(\bigcup_{t\neq
         k} T_t\) is
       connected, \(1\leq k\leq m\).
     \end{itemize}
   \end{lem}

\begin{proof}  (i) Assume \(X_i \cap Z_j \neq\emptyset\) and \(X_s \cap
  Z_j=\emptyset\).  Then \(X_i \cap X_s\), being a subset of the
  singular locus, lies in \(\bigcup_{t\neq j} Z_t\).  Therefore,
  \(X_s\) is disjoint from \(T_j\).  Thus
\[
    T_j=\left[\left(\bigcup_ {X_i\cap Z_j\ne
          \emptyset}X_i\right)\setminus \left(\bigcup_{X_s\cap Z_j=\emptyset}
    X_s\right)\right]\setminus\left(\bigcup_{t\neq
    j}Z_t\right)=\left[X\setminus \left(\bigcup_{X_t\notin
    S_j}X_t\right)\right]\setminus\left(\bigcup_{t\neq
    j}Z_t\right)
\]
is open in \(X\).  Since \(T_j\) is the union of the connected sets
\(Z_j\) and \(X_i\setminus
\bigl(\bigcup_{t\neq j} Z_t\bigr)\) (over those \(i\) such that
\(X_i\cap Z_j\neq\emptyset)\), and each of the latter intersects with
\(Z_j\), it is clear that \(T_j\) is connected.

(ii) Let \(x \in X\).  Assume \(x \in Z\), say \(x \in Z_j\).  Then it is
clear that \(x \in T_j\).  Assume \(x \notin Z\), say \(x \in X_i\).  Take
any \(j\) such that \(X_i\cap Z_j \neq \emptyset\) (such \(j\) exists;
otherwise \(X_i\) is a regular component and we have \(n=1\) and
\(m=0\)), then it is clear \(x \in T_j\).

(iii) is obvious.

(iv) We claim that we can find a permutation \((j_1,\ldots,j_m)\) of
\(\{1,\ldots,m\}\) such that \(T_{j_1} \cup \cdots\cup T_{j_r}\) is
connected for all \(1\leq r\leq m\).  Then the desired result follows
by taking \(k=j_m\).   To prove the claim, we pick \(j_1\)
arbitrarily.  Suppose that we have picked \(j_1,\ldots,j_r\) with \(r<
m\), then it suffices to take \(j_{r+1}\) such that \((T_{j_1}\cup
\cdots\cup T_{j_r}) \cap T_{j_{r+1}}\neq\emptyset\).  Indeed if such
\(j_{r+1}\) could not be found, we would deduce that
\(T_{j_1}\cup\cdots\cup T_{j_r}\) is open and closed in the connected
scheme \(X\) by (ii).  This proves the claim.
\end{proof}

  \begin{thm}\label{structure theorem when singular locus is
    disconnected-covers}
With the above setting, assume moreover that 
\(T_1':=T_2\cup\cdots\cup T_m\) is connected.  Each connected component of
\(T_1\cap T_1'\) is a connected component of the regular scheme
\(X\setminus Z\).  Let \(D_1,\ldots,D_s\) be the connected components
of \(T_1\cap T_1'\).  Then the pullback
functor
\[\Cov(X)\xrightarrow{\ \simeq \
    }\Cov(T_1)\prod_{\prod_{1\leq t\leq
    s}\Cov(D_t)}\Cov(T_1')\]
    is an equivalence.
\end{thm}

\begin{proof}  By Lemma~\ref{devissage lemma} (iii), \(T_1 \cap
  T_1'=(T_1\setminus Z)\cap (T_1'\setminus Z)\).  Since \(T_1\setminus
  Z\) and \(T_1' \setminus Z\) are both unions of some of the irreducible
  components of the regular scheme \(X\setminus Z\), the first
  statement is clear.  The second statement follows readily from \ref{van kampen}.
\end{proof}

Again, we come back to fundamental groups.   With the above setting,
we may and do assume that \(T_1':=T_2 \cup \cdots\cup T_m\) is connected by
Lemma~\ref{devissage lemma} (iv).
Let $d_t\in D_t$. Choose
isomorphisms of fiber functors  $F_{d_1}\simeq \cdots\simeq F_{d_s}$
on $T_1$, and those $F_{d_1}\simeq\cdots\simeq F_{d_s}$ on $T_1'$ , so we get
\[
\pi_1^\pet(T_1,d_1)\simeq \cdots\simeq
\pi_1^\pet(T_1,d_s)
\] 
\[
\pi_1^\pet(T_1',d_1)\simeq \cdots\simeq
\pi_1^\pet(T_1',d_s)
\]

\begin{thm}\label{structure theorem when singular locus is
    disconnected}
  There is an  isomorphism 
\[
\pi_1^\pet(X,d_1)\simeq{\bf VK}\bigl(\pi_1^\pet(T_1,d_{1}),\pi_1^\pet(T_1',d_1);\pi_1^\et(D_1,d_1),\ldots,\pi_1^\et(D_s,d_s)\bigr),
\]
where the relevant maps are $\phi_t:\pi_1^\et(D_t,d_t)\to\pi_1^\pet(T_1,d_t)\simeq\pi_1^\pet(T_1,d_1)$
and $\psi_t:\pi_1^\et(D_t,d_t)\to\pi_1^\pet(T_1',d_t)\simeq\pi_1^\pet(T_1',d_1)$. Moreover, the
canonical maps
$\pi_1^\pet(T_1,d_1)\to\pi_1^\pet(X,d_1)$ and $\pi_1^\pet(T_1',d_1)\to\pi_1^\pet(X,d_1)$ correspond, via this isomorphism, to
the projections of the coproduct. 
\end{thm}
\begin{proof} This follows immediately from \ref{structure theorem when singular locus is
    disconnected-covers} and \ref{van kampen
    machine}.\end{proof}

\begin{rmk}  Observe that the singular locus of \(T_1\) is \(Z_1\) and
  that of \(T_1'\) is \(Z_2\cup\cdots \cup Z_m\).  Therefore, the
  above theorem provides an inductive way to compute pro-\'etale
  fundamental groups in terms of pro-\'etale fundmanental groups of
  connected schemes with connected singular locus.  The latter is
  already handled by \ref{structure theorem when the singular locus is
    connected}.  By further performining an induction on the dimension
  of the singular locus, we can describe the pro-\'etale
  fundamental group of an arbitrary (Nagata, J-2, and Noetherian)
  connected scheme in terms of \'etale fundamental groups of normal schemes.
\end{rmk}

\begin{cor} Let \(\scrC\) be the smallest class of Noohi groups such
  that
  \begin{itemize}
  \item \(\scrC\) contains the \'etale fundamental groups of connected
    normal
    schemes.
  \item \(\scrC\) contains discrete free groups of finite rank.
  \item \(\scrC\) is closed under fiber coproducts and quotients.
  \end{itemize}
 Then the pro-\'etale fundamental group of a Nagata, J-2, Noetherian
 connected scheme lies in \(\scrC\).
\end{cor}


\subsection{An example} 
As an example, we will compute the pro-\'etale fundamental group of an
aribtrary connected curve over a separably closed field,
generalizing \cite[Theorem 1.17]{Elena18}, who treated the stable case.  In
fact, we can handle more than curves. Let's
resume \ref{nagata-J2} and \ref{devissage}.
The key assumption that we now make is the
following:
\begin{itemize}
    \item[] {\it The pro-étale fundamental group of  every connected
        component of \(Z\) or \(\tilde Z\) is trivial.} \end{itemize}
This condition is satisfied if $Z$ is a disjoint union of spectrums of
separably closed fields, e.g. when $X$ is a curve over a separably closed
field. 
Let's choose $x\in
X$, $x_i\in\tilde{X}_i$, and paths $F_x\simeq F_{f(x_1)}\simeq
\cdots\simeq F_{f_{(x_n)}}$. 
Then we have

\begin{thm}\label{curve case} With the above setting, \(\pi_1^{\textup{pro\'et}}(X,x)\) is
  the Noohi coproduct of \(\pi_1^{\textup{\'et}}(\tilde X_i,x_i)\),
  \(i=1,\ldots, n\), and a discrete free group of rank
  \(\tilde{m}-m-n+1\), where $m$ denotes the number of connected
  components of $Z$ and $\tilde{m}$ denotes that of $\tilde{Z}$.
\end{thm}
\begin{proof}
It is a nice exercise to prove the theorem by following the route
provided by
\ref{structure theorem when
singular locus is disconnected} and proceed by induction on the
number  $m$.  However, it is easier to use the following argument.
We will apply Theroem \ref{structure theorem when the singular
    locus is connected-covers},   and argue as in the proof of
    Proposition \ref{van kampen machine}.  Write the
  connected components of \(Z\) as \(Z_1,\ldots,Z_m\), and those of
  \(\tilde Z\) as \(\tilde Z_1,\ldots,\tilde Z_{\tilde m}\).
An object of
  \(\Cov(\tilde X)\times_{\Cov{\tilde Z}}\Cov(Z)\) is
  \((\{A_i\}_{i=1}^n,\{B_j\}_{j=1}^m,\{\lambda_k\}_{k=1}^{\tilde m})\),  where \(A_i\) is a
  \(\pi_1(\tilde X_i,x_i)\)-set, \(B_j\) is a set, and \(\lambda_k\) is
  a bijection from \(A_{i(k)}\) to \(B_{j(k)}\) if \(\tilde Z_k
  \subset \tilde
  X_{i(k)}\) and \(f(\tilde Z_k)=Z_{j(k)}\).

  Using that \(X\) is connected, one can show that we can choose indices \(k_1,\ldots,k_{n+m-1}\) so
  that \(\lambda_{k_1},\ldots,\lambda_{k_{n+m-1}}\) identify
  \(A_1,\ldots,A_n,B_1,\ldots,B_m\) with each other, and hence we will use
  \(A_1\) to represent all of them.  The set \(A_1\) now
  carries the action of \(\pi_1(\tilde X_i,x_i)\), \(i=1,\ldots,n\),
  together with the action of the remaning \(\lambda_k\)'s, now
  regarded as
  bijections from \(A_1\) to itself.  There are \(\tilde m-n-m+1\) of
  them.  This description of  \(\Cov(\tilde X)\times_{\Cov{\tilde
      Z}}\Cov(Z)\) easily implies the theorem.
\end{proof} 



\subsection{Discrete representations}

Let \(\lambda \colon \pi \to G\) be a continuous homomorphism of
topological groups.  We say that \(\lambda\) is {\it discrete} if it
factors as \(\pi\to  G'\to G\) with \(G'\) discrete, or equivalently,
\(\ker(\lambda)\) is open in \(\pi\).  When \(\pi\) is a fundamental
group, it is often interesting to know the discrete homomorphisms.

\begin{lem}\label{coproduct factors through discrete group}
   Let \(\pi_1\) and \(\pi_2\) be Noohi groups and let
   \(\pi_1\coprod\pi_2\) be their Noohi coproduct.  Let
   \(\lambda:\pi_1\coprod \pi_2 \to G\) be a continouous homomorphism
   corresponding to \(\lambda _1\colon \pi_1\to G\), \(\lambda_2\colon
   \pi_2\to
   G\).  Then \(\lambda\) is discrete if and only if \(\lambda_1\) and
   \(\lambda_2\) are both discrete.
\end{lem}
\begin{proof}
  The ``only if'' part is obvious.  Assume that both \(\lambda _1\)
  and \(\lambda _2\) are discrete.  Say \(\lambda _i\) factors as
  \(\pi_i\to G_i'\to G\).  Then clearly \(\lambda \) factors as
  \(\pi_1\coprod \pi_2 \to G_1'\coprod G_2'\to G\), where
  \(G_1'\coprod G'_2\) is the coproduct of \(G_1'\) and \(G_2'\) as
  topological groups.  Since \(G_1'\coprod G'_2\) is discrete,
  \(\lambda\) is discrete.
\end{proof}

\begin{lem} \label{Noohi quotient factors through discrete
    group}
Let \(\pi\) be a Noohi group and let \(\pi ':=(\pi/\langle
H\rangle)^\Noohi\) be a Noohi quotient of \(\pi\).  Consider a
commutative diagram of continuous homomorphisms with \(G\) Hausdorff:
\[
  \xymatrix{\pi\ar@{->}[rd]^\lambda\ar@{->}[d]\\
    \pi '\ar@{->}_{\lambda '}[r]&G}
\]
Then \(\lambda\) is discrete if and ony if \(\lambda  '\) is discrete.
\end{lem}

\begin{proof}
Again the ``only if'' part is obvious.  Assume that \(\lambda \) is
discrete and factorizes as \(\pi \to G'\xrightarrow{\phi}G\).  We may assume that
\(G'\to G\) is injective.  Then \(\ker(\pi\to G')=\ker(\pi\to G)
\supset H\).  Therefore, \(\pi\to G'\) factorizes as
\(\pi\to\pi/\langle H\rangle^c\to G'\) in the category of topological
groups.  Applying the functor \((-)^\Noohi\), we get \(\pi\to \pi
'\xrightarrow{\psi} G'\).  It remains to show \(\phi\circ\psi=\lambda
'\).

By construction \(\phi\circ\psi\) and \(\lambda '\) agree on the image
\(M\) of \(\pi\), which is a dense subset of \(\pi'\).  The map
\((\phi\circ\psi,\lambda ')\colon \pi '\to G\times G\) sends \(M\) to
the diagonal of \(G\times G\), which is closed as \(G\) is Hausdorff.
We conclude that the whole image of \((\phi\circ\psi,\lambda ')\) lies
in the diagonal.  This shows: \(\phi\circ\psi=\lambda  '\).
\end{proof}

\begin{lem}\label{discrete quotient etale case} Suppose $g\colon Y'\to Y$ is a finite surjective map of normal
    integral schemes and $y'\in Y'$ is a geometric point. Set
    $y\coloneqq
    g(y')$.  Consider a
commutative diagram of continuous homomorphisms:
\[
  \xymatrix{\pi_1^\pet(Y',y')\ar@{->}[rd]^{\lambda '}\ar@{->}[d]\\
    \pi _1^\pet(Y,y)\ar@{->}_-{\lambda }[r]&G}
\]
Then \(\lambda\) is discrete if and ony if \(\lambda  '\) is discrete.
\end{lem}

\begin{proof}
Indeed, the following commutative diagram:
    \[
        \begin{tikzpicture}[xscale=2.0,yscale=1.2,bmr/.pic={\draw (0,0)--++(-90:2mm)--++(180:2mm)--++(90:2mm)--++(0:2mm);}]
            \path
                (0,0)     node (F) {$\Gal(K(Y'))$}
                +(0:1.5)  node (star) {$\Gal(K(Y))$}
                ++(-90:1) node (X) {$\pi_1^\et(Y',y')$}
                +(0:1.5)  node (Y) {$\pi_1^\et(Y,y)$};
            \draw[right hook->] (F)--(star);
            \draw[->>] (F)--(X);
            \draw[->] (X)--(Y) node[midway,above,scale=.6]{};
            \draw[->>] (star)--(Y);
        \end{tikzpicture} 
    \]
    implies that $\pi_1^\pet(Y',y')\simeq\pi_1^\et(Y',y')\to \pi_1^\et(Y,y)\simeq
    \pi_1^\pet(Y,y)$ is open. Thus $\Ker(\lambda)$ is open if and only if
    $\Ker(\lambda ')$ is open.
  \end{proof}
  

\begin{prop}\label{when the proetale representation
    factors through a discrete one}
Let $X$ be a connected Nagata Noetherian J-2 scheme, and let
$x\in X$ be a geometric point. Let $\{\eta_i\}_{1\leq i\leq n}$ be the set of
all
generic points of $X$, and let $A\coloneqq \prod_{1\leq i\leq n}K_i$,
where each $K_i$ is a finite extension of $\kappa(\eta_i)$. Let
$\tilde{X}_A$ denote the
normalization of $X$ inside $\Spec(A)$. Let \[\lambda\colon
\pi_1^\pet(X,x)\arr G\] be a continuous map of Hausdorff
topological groups. Then $\lambda$ is discrete
if and only if the induced map
\(
\pi_1^\pet(U,u)\arr G
\)
is discrete for each connected component $U$ of $\tilde{X}_A$ and each
geometric point $u\in U$.
\end{prop}
\begin{proof} We just have to show the ``if'' direction. By
    \ref{discrete quotient etale case}, we may assume that
    \(K_i=\kappa(\eta_i)\) for all \(i\), \emph{i.e.} $\tilde{X}_A=\tilde{X}$. By \ref{structure
theorem when singular locus is disconnected},  \ref{coproduct
factors through discrete group} and \ref{Noohi quotient factors
through discrete group}, we may do an induction to reduce the
proposition to the case where
the  singular locus of $Z\subseteq X_\red$ is connected.  So we are
in the situation of Lemma \ref{structure theorem
    when the singular locus is connected}.

    We want to show that the
    induced
    map  $\pi_1^\pet(Z,x)\to G$ is discrete. This would finish the
    proof by applying \ref{coproduct
factors through discrete group} and Lemma \ref{Noohi quotient factors
through discrete group} again.

There is a ring $C\coloneqq \prod_{j\in J}L_j$,  where each
$L_j$ is a finite extension of the residue
field of a generic point of $Z$, such that the
    normalization $\tilde{Z}_C$ of
    $Z$ in $C$
lifts the inclusion $Z\hookrightarrow X$ to a map $\tilde{Z}_C\to \tilde{X}$.
Thanks to the lifting,  all  the maps from the pro-étale
fundamental groups of the connected components of $\tilde{Z}_C$
to $G$ factor through a discrete quotient. Since the dimension
of $Z$ is smaller than that of $X$, we complete our proof by
induction.
\end{proof}

\begin{rmk} Proposition \ref{when the proetale representation factors through a
    discrete one} can also be proved using \cite[Remark 3.20]{Lara19}. Indeed, after we have reduced to the case
    $\tilde{X}_A=\tilde{X}$, we can directly apply \ref{coproduct
factors through discrete group} and \ref{Noohi quotient factors
through discrete group} together with \emph{loc. cit.}
        to finish the proof, because in \emph{loc. cit.}  $\pi_1^\pet(X,x)$ is
        written as a Noohi
        quotient of the
    Noohi coproduct of the fundamental groups of the $\tilde{X}_i$s and a free
    group.
\end{rmk}

\section*{Acknowledgements}
The first named author's work was a part of the project
``Kapibara'' supported by the funding from the European Research Council (ERC)
under the European Union’s Horizon 2020 research and innovation programme
(grant agreement No. 802787).

The third named author would like to thank the Direct Grant
(No. 4053343) of the College of Science of The Chinese University of Hong Kong
for support.
\printbibliography
\end{document}